\documentclass[onefignum,onetabnum,final]{arxiv_class_file}
\usepackage[english]{babel}
\usepackage[margin = 1.3in]{geometry}

\usepackage{amsmath}  
\usepackage{graphicx} 
\usepackage{amssymb}
\usepackage{bbm}
\usepackage{multirow}
\usepackage{graphicx}
\usepackage{blindtext}
\usepackage{url}
\usepackage{subfig}
\usepackage{algorithm}
\usepackage{algorithmic}
\usepackage{xcolor}

\DeclareMathOperator*{\argmin}{arg\,min}
\newcommand{\calD}{\mathcal{D}}
\newcommand{\calE}{\mathcal{E}}
\newcommand{\calS}{\mathcal{S}}
\newcommand{\calG}{\mathcal{G}}
\newcommand{\calF}{\mathcal{F}}
\newcommand{\calR}{\mathcal{R}}
\newcommand{\calP}{\mathcal{P}}
\newcommand{\calM}{\mathcal{M}}
\newcommand{\calH}{\mathcal{H}}

\newcommand{\R}{\mathbb{R}}
\newcommand{\cc}{\mathrm{c}}

\newcommand{\ex}{\mathrm{ex}}
\newcommand{\wass}{{{\mathcal W}_p}}

\newcommand{\pG}{\mathcal{P_G}}

\newcommand{\rd}{\mathrm{d}}

\newcommand{\KL}{\text{KL}}
\newcommand{\true}{\text{true}}

\title{Inverse Problems Over Probability Measure Space\thanks{Submitted to the editors on \today.
\funding{M.O.~and Y.Y.~were supported by National Science Foundation (NSF) through grant DMS-2409855, Office of Naval Research through
grant N00014-24-1-2088, and Cornell PCCW Affinito-Stewart Grant. Q.L.~was supported in part by NSF grant DMS-2308440. L.W.~was supported in part by NSF grant DMS-1846854 and the Simons Fellowship.}}}

\author{Qin Li\thanks{Department of Mathematics, University of Wisconsin-Madison, Madison, WI (\email{qinli@math.wisc.edu}).}
\and Maria Oprea\thanks{Institute for Science and Technology Austria (\email{moprea@ista.ac.at}), corresponding author.}
\and Li Wang\thanks{School of Mathematics, University of Minnesota Twin Cities, Minneapolis, MN (\email{liwang@umn.edu}).}
\and Yunan Yang\thanks{Department of Mathematics,  Cornell University, Ithaca, NY (\email{yunan.yang@cornell.edu}).}}

\headers{Inverse Problems Over Probability Measure Space}{Qin Li, Maria Oprea, Li Wang, and Yunan Yang}
\date{\today}

\begin{document}
\maketitle

\begin{abstract}
Define a forward problem as $\rho_y=\mathcal G_\#\rho_x$, where the probability distribution $\rho_x$ is mapped to another distribution $\rho_y$ using the forward operator $\mathcal G$. In this work, we investigate the corresponding inverse problem: Given $\rho_y$, how to find $\rho_x$? Depending on whether $\mathcal G$ is overdetermined or underdetermined, the solution can have drastically different behavior. In the overdetermined case, we formulate a variational problem $\min_{\rho_x} \mathcal D(\mathcal G_\#\rho_x, \rho_y)$, and find that different choices of the metric $\mathcal D$ significantly affect the quality of the reconstruction. When $\mathcal D$ is set to be the Wasserstein distance, the reconstruction is the marginal distribution, while setting $\mathcal D$ to be a $\phi$-divergence reconstructs the conditional distribution. In the underdetermined case, we formulate the constrained optimization $\min_{\{\mathcal G_\#\rho_x=\rho_y\}}\mathcal E[\rho_x]$. The choice of $\mathcal E$ also significantly impacts the construction: setting $\mathcal E$ to be the entropy gives us the piecewise constant reconstruction, while setting $\mathcal E$ to be the second moment, we recover the classical least-norm solution. We also examine the formulation with regularization: $\min_{\rho_x} \mathcal D(\mathcal G_\#\rho_x, \rho_y) + \alpha \mathsf R[\rho_x]$, and find that the entropy-entropy pair leads to a regularized solution that is defined in a piecewise manner, whereas the $W_2$-$W_2$ pair leads to a least-norm solution where $W_2$ is the 2-Wasserstein metric. 
\end{abstract}

\section{Introduction}
 {
Inverse problems are concerned with recovering an unknown parameter from a set of observations, given a model that maps parameters to the observed data. With the development of experimental design and data acquisition techniques, new types of observations arising in practice, such as images or recordings, can often be better represented as infinite-dimensional objects. Correspondingly, inverse problems in which the unknown parameter itself is infinite-dimensional have gained increasing attention.
}

 {In this paper, we focus on a specific and important instance of an infinite-dimensional parameter: a probability distribution. We take the parameter space to be $\calP(\Theta)$, the space of all probability distributions on  $\Theta \subseteq \R^m$, and work directly with probability measures. More explicitly, we examine the following problem:}
 {
}
\begin{equation}\label{eqn:main_problem}
    \text{Given data }\rho_y\,,\quad\text{reconstruct }\rho_x\text{ so that }\calG_\#\rho_x\approx\rho_y\,.
\end{equation}
Here, \(\rho_x \in \calP(\Theta)\) is a probability measure of the variable \(x \in \Theta \subseteq \mathbb{R}^m\), with \(\calP(\Theta)\) denoting the collection of all probability measures on the domain \(\Theta\). $\calG$ is a function that maps $x\in\Theta$ to $y=\calG(x)\in\calR\subseteq\mathbb{R}^n$. We denote $\calR$ as the range. The given data \(\rho_y \in \calP(\mathbb{R}^n)\) is a probability measure over the variable \(y \in \mathbb{R}^n\). Notably, $\rho_y$ may have non-trivial mass outside the range $\calR$.
The map \(\calG_\#\) then is the deduced pushforward operator that maps a probability measure in \(\calP(\Theta)\) to a measure in \(\calP(\calR)\).

 {This formulation arises from many applications: in weather prediction, the goal is to infer the distribution of pressure and temperature changes~\cite{gneiting2014probabilistic}; in plasma simulation, one aims to infer the distribution of plasma particles using macroscopic measurements~\cite{Ferron_1998,caflisch2021adjoint}; in experimental design, the objective is to determine the optimal distribution of tracers or detectors to achieve the best measurements~\cite{huan2024optimalexperimentaldesignformulations,jin2024optimaldesignlinearmodels,yu2018scalable}; in optical communication, the task is to recover the distribution of the optical environment~\cite{bracchini2004spatial,Korotkova:11,Borcea2015}. Other applications are found in aerodynamics~\cite{del2022stochastic}, biology~\cite{davidian2003nonlinear,daun2008ensemble,tang2023ensemble}, and cryo-EM~\cite{giraldo2021bayesian,tang2023ensemble}. In all these problems, the sought-after quantity is a probability distribution, density, or measure that matches the given data. Consequently, inverse problems in this stochastic setting are naturally formulated as the inversion for a probability distribution~\eqref{eqn:main_problem}. In some literature, this formulation is also termed stochastic inverse problem~\cite{breidt2011measure,butler2012computational,butler2013numerical,butler2014measure,marcy2022stochastic,li2023differential,white2024building}.}

The problem \eqref{eqn:main_problem} can be viewed as an analog of the inverse problem in linear space (e.g., Euclidean space): Given data $y$, we are to reconstruct $x$ so that $\calG(x)\approx y$.
 {
Although questions of existence and uniqueness are of interest in their own right, this study focuses primarily on characterizing the solutions, when they exist. In order to do so, we define the feasible set,
\begin{equation}\label{eqn:feasible_set}
    \calS=\{\rho_x:\,\calG_\#\rho_x=\rho_y\}\subseteq\calP(\Theta)
\end{equation}
being the collection of all possible $\rho_x$, whom under the pushforward action by $\calG$, agrees completely with $\rho_y$.
}

There are three scenarios:
\begin{itemize}
    \item[$\bullet$]{Unique solution}: The ideal situation occurs when the feasible set contains only one element, which corresponds to the natural inversion of $\rho_y$, and solves the problem~\eqref{eqn:main_problem}.
    \item[$\bullet$]{Overdetermined}: The feasible set is empty (corresponding to non-existence). This would occur if \(\text{supp}(\rho_y)\) is not entirely contained in \(\calR\), and there are elements of $y$ that cannot find its inversion.
    \item[$\bullet$]{Underdetermined}: The feasible set has infinitely many elements (corresponding to non-uniqueness). This can happen if \(\text{supp}(\rho_y)\) is completely inside $\calR$, and there are elements of $y$ that find many choices of inversion.
\end{itemize}

 {These scenarios also appear in the deterministic setting when looking for inversion $x\approx \calG^{-1}(y)$. Building on this connection, we develop an analogy between the inverse problem in Euclidean space and the stochastic inverse problem, viewing the latter as a lift of the former. This correspondence is made possible precisely by the pushforward structure of the forward model. Beyond its conceptual appeal, this analogy is useful for applications: the resulting formulas remain easy to interpret, while also revealing an illuminating connection between the deterministic and stochastic inverse problems.}

Numerically, it is a common practice to formulate an optimization problem for finding this $x$ by adding a regularization term. This numerical strategy is feasible in both overdetermined and underdetermined settings. In the overdetermined setting, adding a regularizer can relax the feasible set constraints and tolerate errors in mismatching, and in the underdetermined setting, the added regularizer helps impose prior knowledge. Likewise, for problem~\eqref{eqn:main_problem}, we also have:
\begin{itemize}
    \item[$\bullet$]{Regularized}: This applies to both overdetermined and underdetermined settings, and we relax the problem by tolerating some errors and adding prior knowledge.
\end{itemize}

We are interested in studying problem~\eqref{eqn:main_problem} in (1).~overdetermined setting, (2).~underdetermined setting and (3).~regularized setting. More specifically, we are to spell out the behavior of the solution to problem~\eqref{eqn:main_problem} formulated in the optimization framework in the above three mentioned settings. We now summarize the findings.


\underline{In the overdetermined case,} \(\text{supp}(\rho_y) \varsubsetneqq \calR\), implying that \(\rho_y(\mathbb{R}^n \setminus \calR) \neq 0\). In this scenario, no measure \(\rho_x\) satisfies \(\calG_\#\rho_x = \rho_y\), making the feasible set $\calS$ completely empty. Intuitively, this reflects the fact that \(\rho_y\) possesses a non-trivial mass outside the range \(\calR\), so a perfect match is unattainable. In this case, we solve the variational problem to determine the best match:    \begin{equation}\label{eqn:variation}  
\rho_x^\ast = \argmin_{\rho_x} \calD(\calG_\#\rho_x, \rho_y)\,,
\end{equation}
for a properly chosen $\calD$. Accordingly, we also define the reconstructed data distribution
\begin{equation}\label{eqn:rho_y^ast}  
\rho_y^\ast=\calG_\#\rho_x^\ast\,.
\end{equation}
This formulation seeks a probability measure \(\rho_x\) such that, when pushed forward by \(\calG\), matches \(\rho_y\) optimally, with the optimality quantified by the misfit function \(\calD\). In this setting, it is straightforward to see that $\rho_y^\ast\neq\rho_y$ but maintains some features of $\rho_y$. 

In particular, the choice of $\calD$ is crucial and emphasizes different features of $\rho_y$. Below, we present two concrete scenarios: 
\begin{itemize}
    \item[--] Setting $\calD$ to be any $\phi$-divergence (also called the $f$-divergence in probability theory), $\rho^\ast_y$ is the conditional distribution of $\rho_y$ over the range of $\calG$;
    \item[--] Setting $\calD$ to be a Wasserstein distance, $\rho_y^\ast$ is the marginal distribution of $\rho_y$.
\end{itemize}
We should note that, at this point, these statements are not yet rigorous. The use of the terms ``conditional'' versus ``marginal'' is intended for intuition only. The precise meaning is described in Sections~\ref{sec:f-div} and~\ref{sec:Wp}, respectively.

\underline{In the underdetermined case}, \(\text{supp}(\rho_y)\subseteq\calR\), and there is at least one element in $\calS$. If $\calS$ has multiple or infinitely many solutions, we need to pick one that makes the most physical sense. To do so, we follow the footsteps of Euclidean space and solve the following constrained optimization problem:
\begin{equation}\label{eqn:least_norm}
\rho_x^\ast=\argmin_{\rho_x \in \calS} \calE[\rho_x]\,.
\end{equation} 
The choice of $\calE$ depends on the specific problem and the user's objectives. As written, \eqref{eqn:least_norm} looks for the optimal choice of $\rho_x$ within the feasible set, which is optimal in the sense characterized by $\calE$. Therefore, the choice of $\calE$ is crucial. Different definitions of $\calE$ promote different properties, leading to solutions with contrasting features. We examine the following two specific definitions of $\calE$ and report:
\begin{itemize}
    \item[--] Setting $\calE$ to be entropy, $\rho^\ast_x$ is piecewise constant function on the level sets of $\calG$.
    \item[--] Setting $\calE=\int|x|^2\rd\rho_x$, $\rho^\ast_x$ extends from the least-norm solution defined in Euclidean space.
\end{itemize}
Once again, the statement is not yet rigorous. The terms ``piecewise constant" and ``least-norm solution" need precise definitions. We discuss details in Section~\ref{sec:entropy} and Section~\ref{sec:moment} respectively.



\underline{In the regularized case,} a regularization term is added to either offset the constraints on the feasible set or to encode prior knowledge, leading to the following formulation:
\begin{align} \label{eq:reg}
  \rho_x^\ast=\argmin_{\rho_x} \calD (\calG_\#\rho_x,  \rho_y) + \alpha \mathsf{R}[\rho_x]\,,
\end{align}
where $\alpha>0$ is the regularization parameter and $\mathsf{R}: \mathcal{P}(\Theta) \rightarrow \mathbb{R}$ is the regularizing functional.

Similar to the two cases above, different choices of the pair $(\calD, \mathsf{R})$ promote different characterizations of the optimizer $\rho_x^*$. We consider the following two scenarios:
\begin{itemize}
    \item [--] Setting $\mathcal{D}$ to be the Kullback--Leibler (KL) divergence and setting $
    \mathsf{R}[\rho_x]$ also the KL divergence against a prescribed distribution $\calM$. The optimal solution $\rho_x^*$ is piecewisely defined, with the Radon--Nikodym derivative $\frac{\rd\rho_x^*}{\rd\mathcal{M}}$ being piecewise constant on the level sets of $\mathcal{G}$.
    
    \item [--] Setting $\mathcal{D}$ to be $p$-Wasserstein distance and $
    \mathsf{R}[\rho_x] := \int |x|^p \, \mathrm{d}\rho_x$, the $p$-th order moment of $\rho_x$, then the optimal distribution $\rho_x^*$ can be deduced by a revised least-norm solution map.
\end{itemize}
Similar to the scenarios above, we are to make these terminologies precise. Details will be presented in Section~\ref{sec:reg-KL} and Section~\ref{sec:reg-Wp}, respectively.

It is clear that in all the cases above, the choice of metrics plays a crucial role in the reconstruction. This is not at all surprising. The same phenomenon holds true in Euclidean space and linear function space. Specifically, in the regularized problem, classical regularizers include Tikhonov regularization~\cite{golub1999tikhonov,engl1996regularization}, total variation (TV) norm~\cite{rudin1992nonlinear} or $L^1$ norm~\cite{candes2006stable} for $\mathsf{R}$ to promote sparsity, while the $L^2$ norm (i.e., mean squared error) is often used as the data fidelity term to account for measurement error.

  {The main objective of this article is to understand the behavior of the solutions to~\eqref{eqn:variation},~\eqref{eqn:least_norm} and~\eqref{eq:reg} under different choices of $\calD$, $\calE$, and $(\calD,\mathsf{R})$ pair, and to establish the six statements mentioned above. We summarize the findings in Table~\ref{tab:summary}. }
\begin{table}[h]
    \centering
    \renewcommand{\arraystretch}{1.3} 
    \begin{tabular}{|c|c|c|}
     \hline
        Overdetermined & $\mathbb{R}^d$ & $\calP$ \\ \hline
        Formulation & $x^\ast=\displaystyle \argmin_x \|\calG(x) - y\|$ & $\displaystyle \rho_x^\ast=\argmin_{\rho_x} \calD(\calG_\#\rho_x, \rho_y)$ \\ \hline
        Result & \shortstack{$\calG(x^\ast) = \mathsf{P}_\calG(y)$\\ (Definition~\ref{def:proj})} & 
        \begin{tabular}{c|c}
            {$\calD=\phi$-divergence} & \textbf{$\calD=W_2$} \\ \hline
            \shortstack{ Conditional\\ (Theorem~\ref{thm:phi_divergence})}& \shortstack
            {Marginal\\(Theorem~\ref{thm:marginal})} \\ 
        \end{tabular} 
        \\ 
        \hline\hline
        Underdetermined & $\mathbb{R}^d$ & $\calP$ \\ \hline
        Formulation & $x^\ast=\displaystyle \argmin_{\calG(x)=y} \|x\|$ & $\displaystyle \rho_x^\ast=\argmin_{\rho_x\in\calS}\calE[\rho_x]$ \\ \hline
        Result & \shortstack{least-norm solution \\(Equation~\eqref{eqn:least_norm_Rd})} & 
        \begin{tabular}{c|c}
            {$\calE=$ entropy} & \textbf{$\calE[\rho_x]=\int|x|^2\rd\rho_x$} \\ \hline
            \shortstack{piecewise constant\\(Theorem~\ref{thm:KL-under})} & \shortstack{least-norm solution\\(Theorem~\ref{thm:moment_under})} \\ 
        \end{tabular} \\ \hline
        \hline
                   Regularization & $\mathbb{R}^d$ & $\calP$ \\ \hline
        Formulation & $\!\!\!\min_x \|\calG(x) \!-\! \! y\|^2 \!+\! \alpha \mathsf{R}(\!x\!)$ & $\displaystyle \rho_x^\ast=\argmin_{\rho_x}D(\calG_\#\rho_x, \rho_y) + \alpha \mathsf{R}(\rho_x)$ \\ \hline
        Result & \shortstack{$ x^\ast = \mathcal{F}(y)$ \\(Definition~\ref{def:proj2})} & 
        \begin{tabular}{c|c}
            {entropy-entropy pair} & {$W_2$-moment pair} \\ \hline
            \shortstack{ $\frac{\rd \rho}{\rd \mathcal M}$ piecewise const \\ (Theorem~\ref{reg-KL})}& \shortstack
            {least-norm solution \\ (Theorem~\ref{thm:moment_regularizer}) } \\ 
        \end{tabular}
        \\
        \hline
    \end{tabular}
    \caption{Six theorems to be proved in this paper. }
    \label{tab:summary}
\end{table}

We emphasize throughout this paper that:
\begin{itemize}
    \item We assume the three problems  (\eqref{eqn:variation},~\eqref{eqn:least_norm} and~\eqref{eq:reg}) are well-posed, in the sense that the solution can be found. This assumption is not trivial, given that the problem is posed in an infinite-dimensional space. While the well-posedness of these problems is interesting in its own right, it is somewhat orthogonal to our main goal, and we will set it aside for now.  {Interested readers are strongly recommend to refer to~\cite{li2024stochastic,craig2026unfoldingwassersteinloss}.}
    \item Additionally, we define a few projection operators (see Definitions~\ref{def:proj} and \ref{def:proj2}, and Equation~\eqref{eqn:least_norm_Rd}), and note that the uniqueness of these projections is not a requirement. If multiple projections exist, we have the freedom to select any one of them.
\end{itemize}

It is possible that these observations have appeared previously in the literature. However, to the best of the authors' knowledge, we have not found it well-documented systematically. We should also note that some of the results in simpler cases were reported in earlier work~\cite{li2024stochastic,li2023differential}. In particular, different recovery (conditional versus marginal) for $\calG=\mathsf{A}$ as an overdetermined linear operator was reported in~\cite{li2024stochastic}. Furthermore, in~\cite{li2023differential}, the authors reported a gradient flow optimization algorithm using the kernel method when $\calD$ takes the form of KL divergence~\cite{benamou2015iterative,di2016measure}.

The rest of the paper is structured as follows. Section~\ref{sec:over} is dedicated to problem~\eqref{eqn:variation}, with Section~\ref{sec:f-div} presenting general results for $\phi$-divergences and Section~\ref{sec:Wp} addressing the counterpart for the Wasserstein distance. Section~\ref{sec:under} explores the underdetermined case and studies problem~\eqref{eqn:least_norm}, with Section~\ref{sec:entropy} examining the problem by setting $\calE$ as an entropy and Section~\ref{sec:moment} focusing on setting $\calE$ as the moment of the distribution. Section~\ref{sec:under-reg} examines the regularized formulation with Section~\ref{sec:reg-KL} and Section~\ref{sec:reg-Wp} respectively, dedicated to entropy-entropy pair and Wasserstein-moment pair. In each section, we end with a subsection documenting the results applied to two toy examples, one linear and one nonlinear, both of which have explicit solutions. These examples highlight the contrast between different measures and their mathematical consequences. The proofs are short enough to be included directly in the main text.
 {Beyond toy examples shown in each section, we also study a PDE-based inverse problem numerically in Section~\ref{sec:numerics} where all six theorems are tested with Wasserstein gradient descent used for running optimization.}

 {\subsection{Related work}
Despite its prevalence in applications, the inverse problem of recovering an input probability measure from an observed output probability measure has received comparatively limited mathematical attention. However, related formulations have appeared in the literature.}

 {The works most closely related to ours are~\cite{breidt2011measure,butler2012computational,butler2013numerical,butler2014measure,marcy2022stochastic,li2023differential,white2024building}, where the problem is often described as a \emph{measure-theoretic approach} or a \emph{stochastic inverse problem}. These studies largely adopt a computational perspective and aim to construct a solution consistent with the observed probability measure. By contrast, recent developments in optimal transport and the geometry of measure-valued mappings make it possible to investigate the structure of the solution set in greater detail. Our emphasis is therefore complementary but distinct: we focus primarily on the mathematical foundations of the inverse problem and on rigorous characterizations of its solutions.}

 {Another seemingly related line of research is the Bayesian formulation of inverse problems~\cite{Dashti2017}, in which uncertainty about an unknown parameter is represented by a posterior probability distribution obtained by combining prior information with a likelihood model. Although the Bayesian and present formulations involve similar mathematical objects, their underlying assumptions, interpretations, and solution techniques differ substantially. In particular, a Bayesian formulation explicitly specifies a prior distribution and a model for measurement error, whereas none of the three formulations considered here, \eqref{eqn:variation}, \eqref{eqn:least_norm}, and~\eqref{eq:reg}, requires prior information or a probabilistic model of measurement error. Consequently, the Bayesian posterior can be characterized analytically via Bayes' rule, whereas our solutions rely on an optimization formulation.}

 {More fundamentally, in a conventional Bayesian inverse problem, probability distributions describe uncertainty about an underlying parameter, which is typically finite- or infinite-dimensional but not itself a probability measure. In our setting, by contrast, the unknown is the probability measure itself. A closer Bayesian analog is provided by hierarchical Bayesian approaches~\cite{gelman1995bayesian, white2024building}, where individual parameters $x_i$ are modeled as samples from a population law $\rho_\eta$ and a hyperprior is placed on the population-level parameter $\eta$. Such methods can therefore also infer an unknown population distribution. The distinction is that hierarchical Bayesian inference produces a posterior over $\eta$, and hence over candidate population laws $\rho_\eta$, whereas our formulation directly selects a measure $\rho_x$ through a variational principle. We refer interested readers to~\cite{breidt2011measure} for a clear discussion of the distinction between these two perspectives. Incorporating prior information and a measurement-error model into our framework would lead to a Bayesian inverse problem on the space of probability measures, that is, to a posterior distribution over probability measures. Although such an extension is possible, and techniques related to those in~\cite{OPinBanach,OTfordynamicIP} may prove useful, it lies beyond the scope of the present work. The closest formulations we are aware of are those studied in~\cite{IPonmeasures}, where the parameter belongs to a space of probability measures, while the data lie in a linear function space. This difference in the data space leads to substantially different mathematical structures and analytical techniques.}

\section{The Characterization of 
Solutions to problem~\eqref{eqn:variation}}\label{sec:over}
In this section, we discuss the overdetermined case, meaning the feasible solution set $\calS$ is empty, for example, due to modeling error or noisy data. As an alternative, we look at the relaxation of the inverse problem by considering a variational framework~\eqref{eqn:variation} and seeking for a $\rho_x$ that minimizes the mismatch between produced data $\calG_\#\rho_x$ and the given $\rho_y$. As a result, different choices of the distance function $\calD$ promote different properties, yielding different optimal solutions accordingly. In this section, we pay specific attention to setting $\calD$ as a $\phi$-divergence for any convex function $\phi$ (see Section~\ref{sec:f-div}) and the Wasserstein distance (see Section~\ref{sec:Wp}).  {We conclude in section~\ref{sec:example_over} with a few toy examples to demonstrate the difference in the solutions caused by the choice of $\calD$. We also numerically study a PDE inverse problem in Section~\ref{sec:numerics-over}.}

\subsection{Reconstruction when $\calD$ is $\phi$-divergence}\label{sec:f-div}

This section presents results when \(\calD\) in~\eqref{eqn:variation}   is a \(\phi\)-divergence. We rewrite~\eqref{eqn:variation}   as follows:
\begin{equation}\label{eq:f_div_min}
\min_{\rho_x \in \calP(\Theta)} \calD_\phi (\calG_{\#} \rho_x \| \rho_y)\,.
\end{equation}
The \(\phi\)-divergence between two probability measures \(P\) and \(Q\) is defined as
\[
\calD_\phi(P \| Q) = \int \phi\left(\frac{\rd P}{\rd Q}\right) \, \rd Q\,,
\]
where \(\frac{\rd P}{\rd Q}\) is the Radon–Nikodym derivative of \(P\) with respect to \(Q\), and \(\phi: \mathbb{R}^+ \to \mathbb{R}\) is a convex function such that \(\phi(1) = 0\). Common examples of \(\phi\)-divergences include the aforementioned KL divergence and the \(\chi^2\)-divergence, by setting $\phi(t)$ to be $t\log(t)$ and $(t-1)^2$, respectively.

We claim the reconstruction recovers the conditional distribution. For that, we need to give a precise definition.

\begin{definition}\label{def:conditional}
The conditional distribution of \(\rho_y\) on the range \(\calR\), denoted by $\rho_{y | \calR}^c$  and $\rho_y(\cdot | \calR)$ interchangeably, is defined as follows:
\[
\rho^\cc_{y|\calR}(B) = \frac{\rho_y(B \cap \calR)}{\rho_y(\calR)}\,, \quad \forall\, \text{measurable set } B\,.
\]
Moreover, \(\rho_{y|\mathbb{R}^n \setminus \calR}^\cc\) is the conditional distribution of \(\rho_y\) on \(\mathbb{R}^n \setminus \calR\), defined in a similar manner.
\end{definition}

As a consequence, \(\rho^\cc_{y|\calR}\) is absolutely continuous with respect to \(\rho_y\), and we have:
\[
\frac{\rd\rho^\cc_{y|\calR}}{\rd\rho_{y}}(y) = \begin{cases}
\frac{1}{\rho_{y}(\calR)}\,, & \text{if } y \in \calR\,, \\
0\,, & \text{if } y \notin \calR\,.
\end{cases}
\]

We are now ready to present our first theorem:

\begin{theorem}\label{thm:phi_divergence}
Assume that the variational problem~\eqref{eq:f_div_min} admits a minimizer \(\rho_x^* \in \calP(\Theta)\). Then, we have
\[
\calG_{\#} \rho_x^* = \rho^\cc_{y|\calR}\,.
\]
\end{theorem}

The proof of Theorem~\ref{thm:phi_divergence} requires the use of Measure Disintegration Theorem~\cite[Thm.~5.3.1]{ambrosio2005gradient}, which guarantees the existence of the conditional distribution. Specifically, for any given map \( T \) that maps from a probability space \( (Y, \mathcal{B}_Y, \mu) \) to a measurable space \( (Z, \mathcal{B}_Z) \), and defining \(\nu = T_\# \mu\), the theorem states that for \(\nu\)-a.e. \(z \in Z\), there exists a family of probability measures \(\{\mu_z : z \in Z\}\) on \( (Y, \mathcal{B}_Y) \) that satisfies: $\mu(B) = \int_Z \mu_z(B) \, \rd\nu(z)$ for any measurable set $B \in \mathcal{B}_Y$. The collection \(\{\mu_z\}_z\) is called the disintegration of \(\mu\) with respect to \(T\). In our context, $Y$ would be the whole of $\R^n$ and $\mu$ is our data distribution $\rho_y$. We need to identify the appropriate measurable space and find the correct map \(T\), as detailed in the following proof.

\begin{proof}[Proof of Theorem~\ref{thm:phi_divergence}]
    Define a map \( T: \mathbb{R}^n \rightarrow \{0,1\} \) such that 
    \[
    z = T(y) = \begin{cases}
        1 \,, &y \in \calR\,,\\
        0 \,, & y \not\in \calR\,.\\
    \end{cases}
    \]
    By applying the Measure Disintegration Theorem to \( \rho_y \) based on the map \( T \), we obtain a discrete probability measure \( \nu \), with 
    \[
    \nu(1) = \rho_{y}(\calR), \quad \nu(0) = \rho_{y}(\mathbb{R}^n \setminus \calR),
    \]
    and the following disintegration of \( \rho_y \):
    \begin{equation}
       \rho_{y} = \nu(1) \, \rho_{y|\calR}^\cc + \nu(0)  \,  \rho_{y|\mathbb{R}^n \setminus \calR}^\cc\,,
    \end{equation}
    where \( \rho_{y|\calR}^\cc \) and \( \rho_{y|\mathbb{R}^n \setminus \calR}^\cc \) are the conditional distributions given in Definition~\ref{def:conditional}. The Measure Disintegration Theorem further states that this disintegration is unique.
   
To show that \( \rho_{y|\calR}^\cc \) is the optimal solution, we rewrite the variational problem~\eqref{eq:f_div_min} as:
\begin{equation}\label{eq:f_div_min_2}
\min_{\rho_x \in \calP(\Theta)} \calD_\phi (\calG_{\#} \rho_x || \rho_y) = \min_{\substack{\rho_y' \in \calP(\R^d) \\ \rho_y'({\R^n\setminus\calR})=0 }} \calD_\phi (\rho_y' || \rho_y)\,.
\end{equation}
This is true because $\{\calG_\# \rho_x |\rho_x\in \calP(\Theta)\}   = \{\rho_y'\in \calP(\R^d): \rho_y' (\R^n \setminus \calR) = 0\}$. The ``$\subseteq$'' direction holds directly since $\calR = \calG(\Theta)$.  The ``$\supseteq$'' direction holds because, for a given $\rho_y'$, $\text{supp}(\rho_y')\subseteq \calR$ and by using the left inverse function of $\calG$ 
we can define a distribution $\rho_x' \in \calP(\Theta)$ such that $\calG_\# \rho_x' = \rho_y'$.

Without loss of generality, we only examine $\rho_y'$ that is absolutely continuous with respect to $\rho_y$\footnote{Otherwise, $\calD_\phi(\rho_y' || \rho_y)$ achieves the maximum value, making such $\rho_y'$ irrelevant as we aim to find the minimum. The maximum value for this divergence is $\infty$ in the case of KL and $\chi^2$, and $1$ in the case of TV.}. By the definition of the \( \phi \)-divergence, we have
    \begin{eqnarray*}
    \calD_\phi(\rho_y' || \rho_y) &=& \int \phi\left(\frac{\rd\rho_y'}{\rd\rho_y}\right) \rd\rho_y\\
        &=& \nu(1) \int \phi\left(\frac{\rd\rho_y'}{\rd\rho_y}\right) \rd\rho_{y|\calR}^\cc +  \nu(0) \int \phi\left(\frac{\rd\rho_y'}{\rd\rho_y}\right) \rd\rho_{y|\mathbb{R}^n\setminus \calR}^\cc\,.
    \end{eqnarray*}
Since \( \rho_y'(\mathbb{R}^n \setminus \calR) = 0 \), the second term is a constant $\nu(0)\phi(0)$. Thus, we are left with:
\begin{eqnarray*}
    \calD_\phi(\rho_y' || \rho_y) &=& \nu(1) \int \phi\left(\frac{\rd\rho_y'}{\nu(1)  \rd\rho_{y|\calR}^\cc + \nu(0)  \rd\rho_{y|\mathbb{R}^n\setminus \calR}^\cc }\right) \rd\rho_{y|\calR}^\cc + \nu(0)\phi(0)\\
    &=& \nu(1) \int \phi\left(\frac{1}{\nu(1) }\frac{\rd\rho_y'}{ \rd\rho_{y|\calR}^\cc} \right) \rd\rho_{y|\calR}^\cc + \nu(0)\phi(0) \\
    &\geq &\nu(1) \phi\left(\frac{1}{\nu(1)}\right) + \nu(0)\phi(0)\,,
\end{eqnarray*} 
where in the last step we applied Jensen's inequality, leveraging the convexity of \( \phi \). The equality holds when \( \rho_y' = \rho_{y|\calR}^\cc \), completing the proof.
\end{proof}
Although the reconstructed $\calG_\# \rho_x^*$ is expected to somewhat agree with $\rho_y$ on the range $\calR$, the fact that the mismatch between $\rho_y$ and  $\calG_\# \rho_x^*$ on $\calR$ is merely a constant multiple is not entirely trivial. Jensen's inequality and the convexity of $\phi$ play a major role.

\subsection{Reconstruction when $\calD$ is $W_p$}\label{sec:Wp}

We now move on by setting \( \calD \) as a Wasserstein distance. This amounts to rewriting~\eqref{eqn:main_problem} as:
\begin{equation}\label{eq:W_min}
\inf_{\rho_x \in \calP(\Theta)} \wass (\calG_{\#} \rho_x,  \rho_y),
\end{equation}
where \( \wass(\cdot, \cdot) \) is the \( p \)-Wasserstein distance. For any two probability measures \( \mu \) and \( \nu \), the \( p \)-Wasserstein distance is:
\begin{equation}\label{eq:wass_p}
\wass(\mu, \nu) = \left(\min_{\gamma \in \Gamma(\mu, \nu)} \int d^p(x, y) \,\mathrm{d}\gamma\right)^{1/p}\,, \quad p \geq 1\,,
\end{equation}
where \( d \) is a metric over \( \mathbb{R}^n \) and \( \Gamma(\mu,\nu) \) represents the set of all couplings between the two measures $\mu$ and $\nu$. The most common choice of $d$ is the Euclidean distance.

To characterize the minimizer of \eqref{eq:W_min}, we first define an inversion map $\calF$ and its associated projection map \( \pG \) as follows. 
\begin{definition} \label{def:proj}
    Define the inversion map \( \calF: \mathbb{R}^n \rightarrow \Theta \) as 
    \[
    \calF(y^*) = \argmin_{x\in\Theta} d(\calG(x),y^*)\,,
    \]
    and the projection map as
    \[
    \pG(y^*) = \argmin_{y \in \calR} d(y,y^*) = \calG(\calF(y^*))\,.
    \]
    If the minimizer is not unique, one has the freedom to select one.
\end{definition}
\begin{theorem}\label{thm:marginal}
Assume solution to problem~\eqref{eq:W_min} exists. Then one of the minimizers takes the following form:
\begin{equation}\label{eqn:push_F}
\rho_x^*=\calF_{\#}\rho_y\,,
\end{equation}
and the reconstruction recovers the projection:
\[
\calG_{\#} \rho_x^*  = \pG_{\#} \rho_y\,.
\]
\end{theorem}
   {As is apparent from the theorem, the relationship between the reconstructed and reference data resembles its Euclidean counterpart: the reconstructed data are obtained by projection onto the range of the forward operator. This analogy arises from the particle-level geometry encoded by the Wasserstein metric. In particular, for Dirac measures, the Wasserstein distance reduces to the underlying ground-space distance. This geometric correspondence is reflected in the structure of the optimizer. By contrast, the entropy-based discrepancy considered in Theorem~\ref{thm:phi_divergence} has no analogous pointwise geometric interpretation, making the resulting optimizer less directly interpretable.}

\begin{proof}
    By the definition of the \( \wass \) distance, we have 
    \begin{align} \label{p1}
        \wass(\calG_{\#} \rho_x, \rho_y)^p = \int d(\tilde{y}, y)^p \, \rd \pi(\tilde{y}, y)\,,
    \end{align}
    where \( \pi \in \Gamma(\calG_{\#} \rho_x, \rho_y) \) is the optimal coupling between \( \calG_{\#} \rho_x \) and \( \rho_y\). From Definition~\ref{def:proj}, we can deduce that 
    \begin{align}
    \int d(\tilde{y}, y)^p \, \rd \pi(\tilde{y}, y) &\geq 
    \int d ( \pG(y), y)^p \, \rd \pi(\tilde{y}, y) \nonumber \\
    & = \int d ( \pG(y), y)^p \, \rd \rho_y(y)\nonumber 
    \\ 
    & =  \int d(z, y)^p \, \rd \gamma(z, y),\quad \text{where }  \gamma = \widetilde{\pG}_\# \rho_y \text{ and } \widetilde{\pG}= \begin{pmatrix}
        \pG \\ 
         \mathsf{I}_n
    \end{pmatrix} \nonumber \\
    & \geq \wass(\pG_{\#} \rho_y, \rho_y)^p \,,\label{p2}
    \end{align}
    where $ \mathsf{I}_n$ is the $n$-dimensional identity matrix.
    Considering
    \[
    \pG_{\#} \rho_y=\left(\calG\circ\calF\right)_{\#}\rho_y =\calG_{\#}\left(\calF_{\#}\rho_y\right)=\calG_{\#}\rho_x^*\,,
    \]
    and recalling definition~\eqref{eqn:push_F}, the inequality above shows that \( \wass(\calG_{\#} \rho_x, \rho_y)^p \geq \wass(\calG_{\#} \rho^*_x, \rho_y)^p\) for all $\rho_x$, concluding the proof.
\end{proof}


\subsection{Examples}\label{sec:example_over}
A key distinction between the $\phi$-divergence case and the $p$-Wasserstein case lies in how the reference data distribution $\rho_y$ is utilized. In the $\phi$-divergence case, the optimizer relies only on a small subset of the reference distribution $\rho_y$: $\rho_y$ confined within $\calR$. In contrast, the $p$-Wasserstein case employs the entire reference distribution $\rho_y$ to generate the marginal distribution. Neither result is completely surprising, but their contrasting features bring out beautiful effects that we did not find in the literature. We demonstrate these results using a couple of simple examples.

\subsubsection{Linear pushforward maps}
Assume $\calG = \mathsf{A}: \Theta= \mathbb{R}^m \rightarrow \mathbb{R}^n$ is linear and full-rank. Since we are in the overdetermined setting, $m<n$. Then the range $\calR = \mathsf{A}(\Theta) \cong \mathbb{R}^{m}$. 
If the given data distributions $\rho_y \in \mathcal{P}(\mathbb{R}^n)$ has nontrivial support outside $\calR$, the feasible set $\calS$ is empty. We then consider problem~\eqref{eqn:variation} to find the optimal solution $\rho_x$.
\begin{itemize}
    \item[--] $\calD$ is a $\phi$-divergence: Theorem~\ref{thm:phi_divergence} states that $\mathsf{A}_\#\rho_x^*$ recovers the conditional distribution of $\rho_y$, namely:
    \[\mathsf{A}_\#\rho_x^*=\rho_y(\cdot|\calR)\,.
    \]
\item[--] $\calD$ is the $p$-Wasserstein metric: Theorem~\ref{thm:marginal} states that $\mathsf{A}_\#\rho_x^*$ recovers the marginal distribution of $\rho_y$. Indeed, in this situation, one can even compute $\rho_x^\ast$ (see~\cite{li2023differential}):
\[
\rho_x^\ast=\mathsf{A}^\dagger_\# \rho_y\,,
\]
where $\mathsf{A}^\dagger = (\mathsf{A}^\top \mathsf{A})^{-1} \mathsf{A}^\top$ is the pseudoinverse (left inverse) of $\mathsf{A}$. 
Consequently, $\mathsf{A}_\#\rho_x^\ast=\left( \mathsf{A}\mathsf{A}^\dagger\right)_\# \rho_y$ indeed recovers the marginal distribution of $\rho_y$ along the $m$-dimensional subspace of $\mathbb{R}^n$.
\end{itemize}

\subsubsection{A simple nonlinear example}
Consider the map $\calG: [0,1]\times[0,2\pi]\to B_{(0,0)}(1)$ where
\[
(y_1,y_2)=\calG(r,\theta)=(r\cos\theta\,,r\sin\theta)\,.
\]
We prepare our noisy data as:
\[
\rho_y = \frac{1}{2\pi}\mathbb{I}_{y_1^2 + y_2^2 < 1}+\frac{1}{2\pi}\delta(y_1^2 + y_2^2 - 4)\,,
\]
where $\mathbb{I}_{A}$ is the indicator function on the set $A$. This means that we prepare half of our data as a uniform distribution over the ball $B_{(0,0)}(1)$ in the range (as denoted by $\frac{1}{2\pi}\mathbb{I}_{y_1^2 + y_2^2< 1}$), and the other half of the data uniformly on the ring of radius $2$ (as denoted by $\frac{1}{2\pi}\delta(y_1^2 + y_2^2-4)$, with an extra $\frac{1}{2\pi}$ factor coming from normalization\footnote{Note that the integral of $\delta(y_1^2 + y_2^2-R^2)$ over the entire plane is $\pi$ for all values of $R>0$.}. See Figure~\ref{fig:over-data} for an illustration. According to our earlier statement, conditional and marginal distributions are obtained if $\phi$-divergence or the $p$-Wasserstein distance is used, respectively. Namely:
\begin{itemize}
    \item[--]{$\calD$ is a $\phi$-divergence:}
    \[
\rho_{y}^\ast=\calG_\# \rho_x^\ast = \frac{1}{\pi}\mathbb{I}_{y_1^2 + y_2^2\leq 1}\,,
    \]
    This is the recovery of the conditional distribution. All data points, conditioned on them being in the range $B_{(0,0)}(1)$, are preserved. All data points outside the range are completely forgotten. Note that the normalizing constant has changed. See Figure~\ref{fig:over-f-div} for the plot of $\rho_y^\ast$.
    \item[--]{$\calD$ is the $W_p$:}
    \[
\rho_{y}^\ast= \calG_\# \rho_x^\ast = \frac{1}{2\pi}\mathbb{I}_{y_1^2 + y_2^2 < 1}+\frac{1}{2\pi}\delta(y_1^2 + y_2^2- 1)\,.
    \]
    This is the recovery of the marginal distribution, with all data points projected to the range $B_{(0,0)}(1)$. In this context, all data points sitting on the ring of $\delta(y_1^2 + y_2^2- 4)$  are projected down to the ring of $\delta(y_1^2 + y_2^2-1)$, within the range. See Figure~\ref{fig:over-Wp} for the plot of $\rho_y^\ast$.
\end{itemize}


\begin{figure}
    \centering
    \subfloat[Data distribution $\rho_y$]{\includegraphics[height = 3.6cm]{ 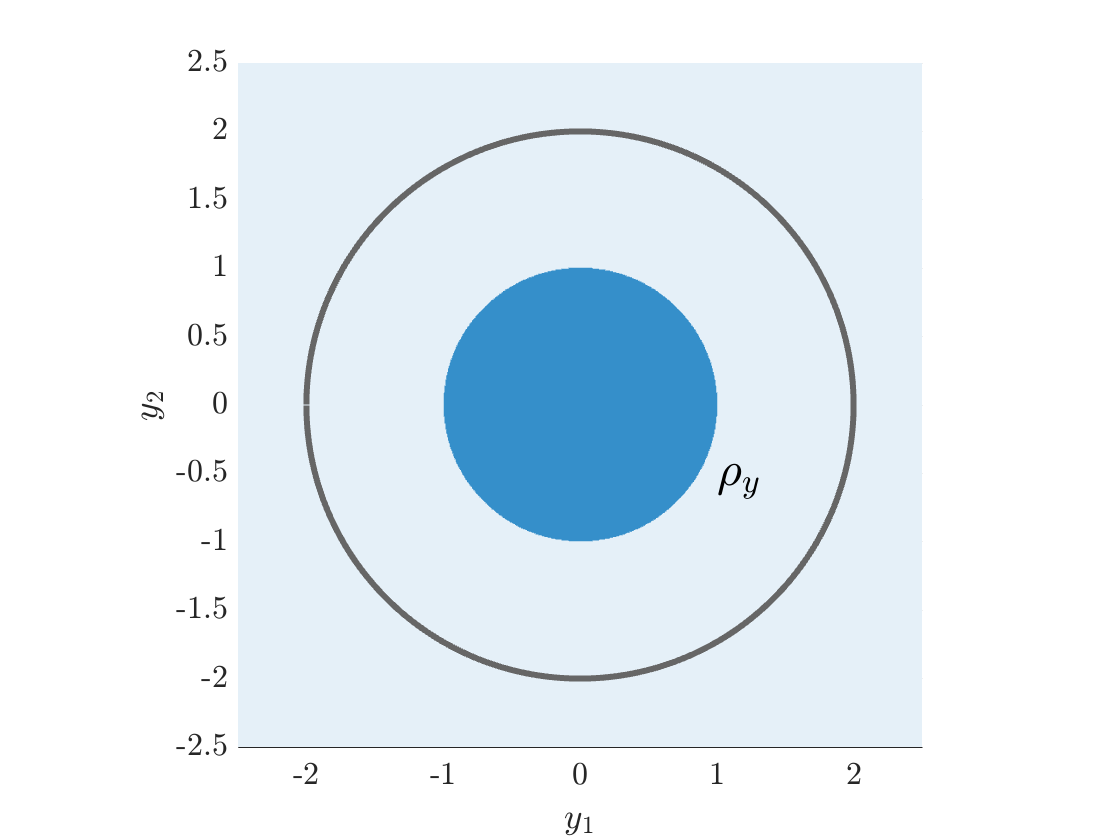}\label{fig:over-data}}
    \subfloat[$\phi$-divergence optimizer]{\includegraphics[height = 3.6cm]{ 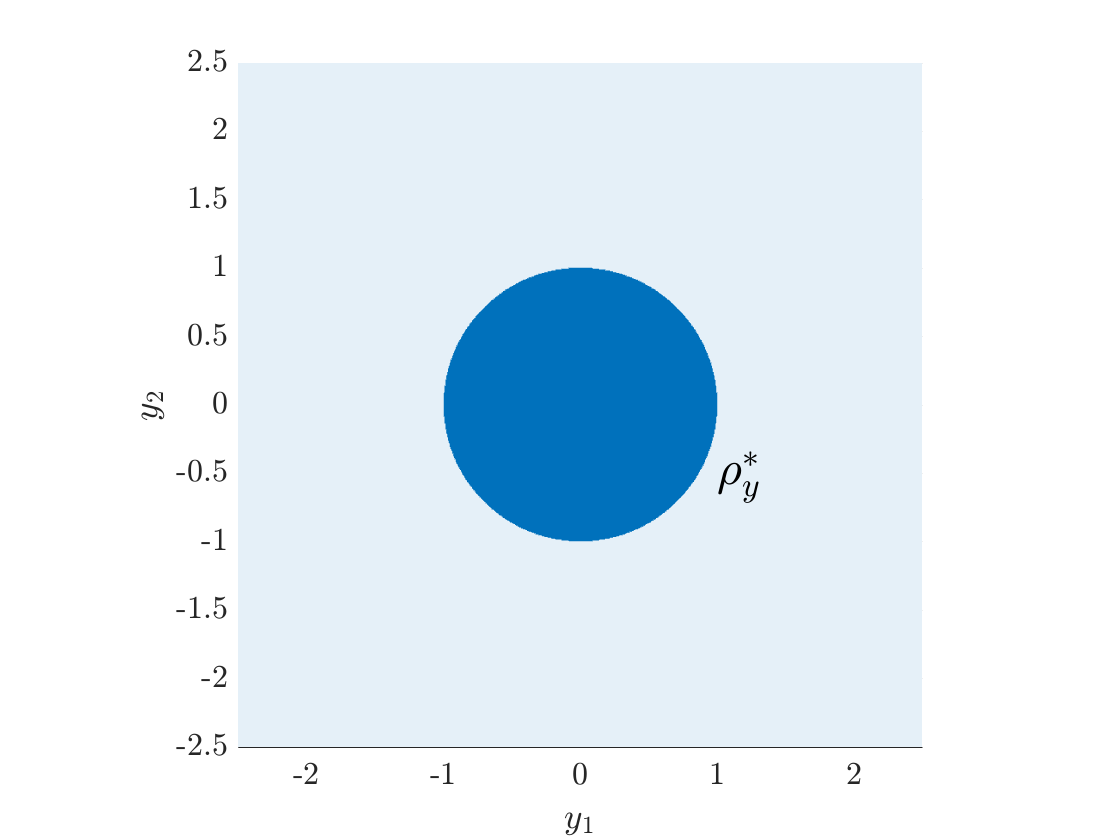}\label{fig:over-f-div}}
    \subfloat[$W_p$ optimizer]{\includegraphics[height = 3.6cm]{ 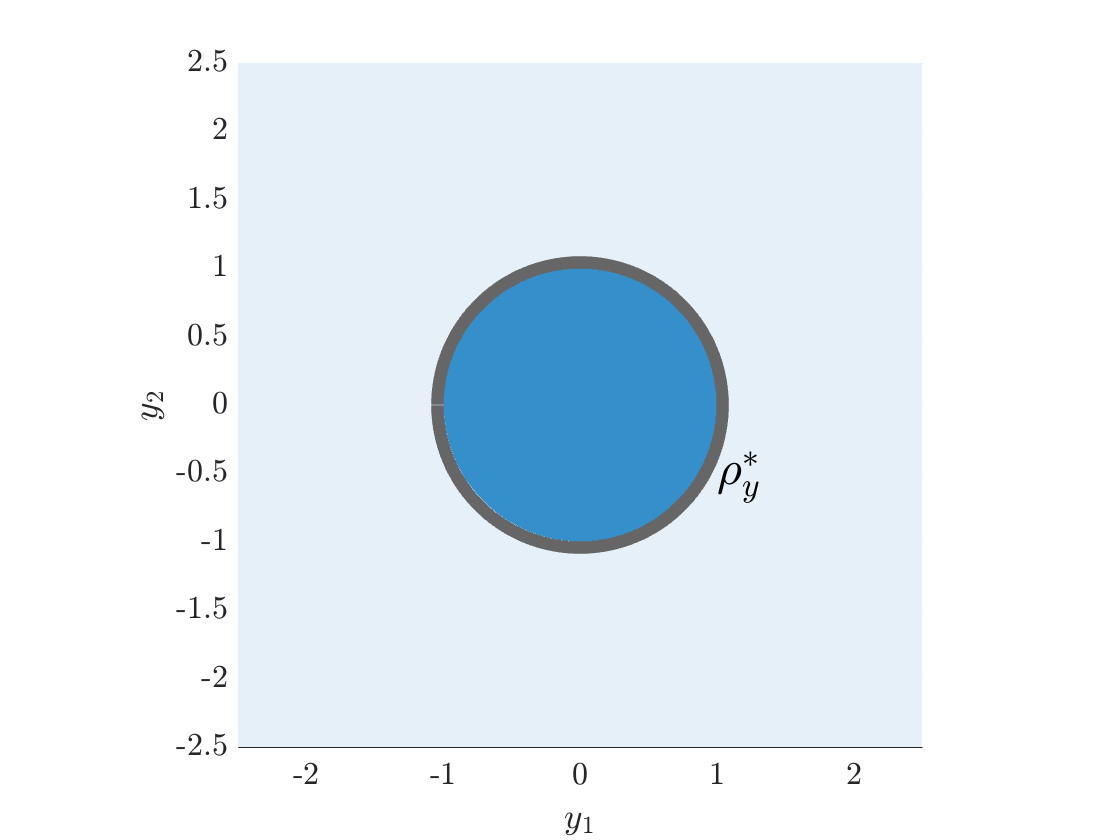}\label{fig:over-Wp}}
    \caption{We denote $\rho_y^* := \calG_\#\rho_x^*$; (a): the noisy data distribution with mass supported outside the range $\calR$; (b): the reconstructed data distribution with $\calD$ being the $\phi$-divergence; (c): the reconstructed data distribution with $\calD$ being the $p$-Wasserstein metric.}
    \label{fig:over}
\end{figure}

\section{The Characterization of 
Solutions to problem~\eqref{eqn:least_norm}}\label{sec:under}
In this section, we discuss the underdetermined situation. This is the situation where there is more than one element in the feasible solution set $\calS$, and we are tasked to pick one by solving~\eqref{eqn:least_norm}. Naturally, different choices of $\calE$ will promote different properties and produce different optimal solutions accordingly. As a natural start, we pay special attention to the settings when $\calE$ is defined as the entropy (see Section~\ref{sec:entropy}) or the second-order moment of the distribution (see Section~\ref{sec:moment}).  {There are some simple examples when we have analytical solutions, and we illustrate them in Section~\ref{sec:example_under}. We numerically study a PDE-based inverse problem in Section~\ref{sec:numerics-under}.}

\subsection{Optimization when $\mathcal E[\rho_x] = \int \rho_x \ln \rho_x \rd x$}\label{sec:entropy}
We first examine problem~\eqref{eqn:least_norm} by setting $\calE$ as the entropy. As such, we look for $\rho_x$ that are absolutely continuous with respect to the Lebesgue measure, which means that we consider $\calP_{\text{ac}}(\Theta)$ to be our search space. For the convenience of the derivation, we also assume throughout this section that $\rho_y\in \calP_{\text{ac}}(\calR)$, and that $\Theta\subsetneqq \R^m$ is compact.

The optimal solution can be explicitly written down as given in Theorem~\ref{thm:KL-under} below.
\begin{theorem} \label{thm:KL-under}
The optimizer $\rho_x^*$ of problem~\eqref{eqn:least_norm} with $\calE[\rho_x] = \int \rho_x \ln \rho_x \rd x$ has the following property:
\begin{align} \label{rho_x_star_under}
\rho_x^*(\, \boldsymbol{\cdot}\, |{\calG^{-1}(y)}) = \mathcal{U}\left( {\calG^{-1}(y)}\right)\,,\quad\forall y\,,
\end{align}
where $\rho_x^*(\cdot|{\calG^{-1}(y)})$ denotes the conditional distriution of $\rho_x^*$  on the set $\calG^{-1}(y)$ and $\mathcal{U}\left( {\calG^{-1}(y)}\right)$ is the uniform distribution over the set  ${\calG^{-1}(y)}$. In particular:
\begin{align} \label{312}
\quad \rho_x^*(x)=\frac{\rho_y(\calG(x))}{\int\delta(\calG(x)-\calG(x'))\rd x'}\,.
\end{align}
\end{theorem}

The theorem  states that the mass cumulated at $y=\calG(x)$ is transferred back to the $x$ domain and distributed uniformly across the preimage $\calG^{-1}(y)$ by all the points in $x$. This property aligns with the fact that entropy increases with variations, so minimizing entropy discourages variations. The uniform distribution, having the least variation, minimizes the entropy.

\begin{proof}[Proof of Theorem~\ref{thm:KL-under}]
To begin with, let $\rho_x^*$ be the optimizer of problem~\eqref{eqn:least_norm} with $\calE[\rho_x] = \int \rho_x \ln \rho_x \rd x$. We claim that there is a function $g: \calR \rightarrow \mathbb{R}_{\geq 0}$ so that 
\begin{equation}\label{eq:g}
\rho_x^*(x) = g(\calG(x))\,.
\end{equation} 
To see this, consider the Lagrangian to the constrained optimization problem~\eqref{eqn:least_norm}:
\begin{align*}
    \mathcal L[\rho_x] &:= \int \rho_x(x)\ln \rho_x(x) \rd x - \int \lambda(y) \left((\calG_\# \rho_x)(y) - \rho_y(y) \right) \rd y
    \\ & = \int \rho_x(x) \ln \rho_x(x) \rd x - \int \lambda(\calG(x))  \rho_x(x) \rd x +\int \lambda(y) \rho_y (y) \rd y\,,
\end{align*}
where $\lambda(y)$ is the Lagrangian multiplier. The optimizer $\rho_x^*$ satisfies the first-order optimality condition
\begin{align*}
    \left.\frac{\delta \mathcal L }{\delta \rho_x} \right|_{\rho_x^*} = \ln \rho_x^*(x)  + 1 - \lambda(\calG(x)) =  C\,,
\end{align*}
which leads to $\ln \rho_x^*(x) = \lambda (\calG(x))-1 +C $ where the constant $C$ is determined by normalization\footnote{Note that $\left.\frac{\delta \mathcal L }{\delta \rho_x} \right|_{\rho_x^*}=0$ in this context means $\langle \left.\frac{\delta \mathcal L }{\delta \rho_x} \right|_{\rho_x^*}\,,\rho-\rho_x^*\rangle = 0$ for all $\rho$.}. This verifies our claim in~\eqref{eq:g} as $\rho_x^*(x)$ only depends on the value of $\calG(x)$, which also leads to~\eqref{rho_x_star_under}. 

In order to determine an explicit formula for $g$, we use the fact that the constraint has to be satisfied, i.e., $\calG_\#\rho_x^* =\rho_y$. Consider an arbitrary test function $f \in \mathcal{C}_c^\infty (\R^n)$. Then the constraint can be rewritten as:
\begin{align*}
    \int f(y) \rho_y(y) \rd y &= \int f(y) \rd \left(\calG_\#\rho_x^*\right) (y) \\
    & = \int f(\calG(x))  \rho_x^* \rd x \qquad \text{where}\quad  \rho_x^*  = g(\calG(x))\\
    &= \int f(\calG(x)) g(\calG(x)) \rd x \\
    &= \int f(y) g(y) \int \delta\left(y - \calG(x)\right) \rd x\,.
\end{align*}
Since the above holds for any test function $f$, we conclude that
$$ 
g(y) \int \delta(y - \calG(x)) \rd x = \rho_y(y) \implies g(y) = \frac{\rho_y(y)}{\int \delta(y - \calG(x)) \rd x}\,.
$$
Plugging the above expresion into~\eqref{eq:g}, we obtain~\eqref{312}.
\end{proof}


\subsection{Optimization when $\calE[\rho_x]=\int|x|^2\rd\rho_x$}\label{sec:moment}
Just as in the deterministic case where the least-norm solution is sought for when the feasible set is not unique, in this section, we look for the ideal distribution in the feasible set $\calS$ by choosing one that has the least ``norm". In particular, we will set the objective function to be the second-order moment of the distribution. To ensure the second-order moment even exists, we naturally work with the space of $\mathcal{P}_2(\Theta)$.

Similar to the result presented in Section~\ref{sec:entropy}, we can explicitly write down the optimal solution given in Theorem~\ref{thm:moment_under} below.
\begin{theorem}\label{thm:moment_under}
The optimizer $\rho_x^*$ of problem~\eqref{eqn:least_norm} with $\calE[\rho_x] = \int |x|^2 \rho_x \rd x$ is 
\begin{align*}
     \rho_x^* = \calH_\#\rho_y\,, 
\end{align*}
where $\calH(y)$ is the minimal-norm solution to $\calG(x) = y$ in the Euclidean space:  
\begin{equation}\label{eqn:least_norm_Rd}
     \calH(y) := \argmin_{\calG(x) = y,\, x\in\Theta} |x|^2\,.
\end{equation}
If the minimimum is not unique, we have the freedom to choose one.
\end{theorem}
   {Once again, this result parallels the least-norm solution in Euclidean geometry. Indeed, the second moment of a probability measure \(\rho\) is precisely its squared Wasserstein-2 distance to \(\delta_0\), and this correspondence is reflected in the structure of the optimizer. By contrast, entropy as a selection criterion has no direct analogue in the pointwise Euclidean setting; consequently, Theorem~\ref{thm:KL-under} has no direct Euclidean counterpart.}

\begin{proof}
    We first note that, from the definition of $\calH$:
    \begin{align}
        |\calH(\calG(x))|^2 \leq |x|^2\,.
    \end{align}
    Then to show that $\rho_x^*$ is optimal, we see that for any other $\rho_x$ such that $\calG_\# \rho_x = \rho_y$, we have:
    \begin{align*}
        \int |x|^2 \rd\rho_x^*(x) & = 
        \int |x|^2 \rd (\calH_\# \rho_y ) = \int |\calH(y)|^2\rd\rho_y (y)
        \\ & = \int |\calH(y)|^2 \rd \left( \calG_\#\rho_x \right) (y)
        = \int |\calH(\calG(x))|^2 \rd\rho_x(x)
        \\ & \leq \int |x|^2 \rd\rho_x(x)\,.
    \end{align*}
\end{proof}


\subsection{Examples}\label{sec:example_under}
Both proofs are straightforward, but the difference between the two theorems signals the drastic differences between moment-minimization and entropy-minimization. According to Theorem~\ref{thm:KL-under}, minimizing entropy encourages the pull-back samples to spread out the whole preimage. On the contrary, Theorem~\ref{thm:moment_under} suggests that the solution that minimizes moments pulls every $y\in\calR$ back to one sample in the preimage, defined by $\calH$.

\subsubsection{Linear pushforward maps}
If $\calG$ is linear with full column rank and we denote the pushforward map by a matrix $\mathsf{A}:\Theta \subset \mathbb{R}^m \rightarrow  \mathbb{R}^n$ where $\Theta $ is a compact subset. We assume $\mathsf{A}$ is full-rank, $m > n$ and $\calS$ contains infinitely many elements. In this context, the range $\calR = \mathsf{A}(\Theta)\subset\mathbb{R}^n$, and we assume the support of the data distribution $\text{supp}(\rho_y) \subseteq \calR = \mathsf{A}(\Theta)$. To compute the preimage, we note that for every fixed $y\in \mathbb{R}^n$, $\mathsf{A}^{-1}(y)$ is a compact subset lying in a linear subspace of at most  $(m-n)$-dimensional. Moreover, define the least-norm solution:

\begin{equation}\label{eqn:min_norm_A}
    x_y=\argmin_{\mathsf{A}x = y} |x|^2 =  \mathsf{A}^\top(\mathsf{A} \mathsf{A}^\top)^{-1} y = \mathsf{A}^\dagger y\,,
\end{equation}
where $\mathsf{A}^\dagger$ denotes the pseudoinverse (i.e., right inverse) of matrix $\mathsf{A}$. 
Then the preimage of $y$ can be represented as:
\[
\mathsf{A}^{-1}(y) = \{x_y+\mathsf{A}_\perp\}\cap\Theta\,,
\]
where $\mathsf{A}_\perp$ is the subspace perpendicular to the row-space of $\mathsf{A}$. We further define, according to~\eqref{eqn:least_norm_Rd}:
\[
\calH(y)=\argmin\{|x|^2: x\in \mathsf{A}^{-1}(y)\}\,.
\]
\begin{itemize}
    \item[--] When $\calE = \int \rho \log \rho\, \rd x$, based on Theorem~\ref{thm:KL-under}, the optimal distribution $\rho_x^\ast$ conditioned on the preimage $\mathsf{A}^{-1}(y)$ is the uniform measure over $\mathsf{A}^{-1}(y)\subset \Theta$.
    \item[--] When $\calE = \int |x|^2 \rd\rho(x)$, based on Theorem~\ref{thm:moment_under}, the optimal distribution $\rho_x^\ast$ conditioned on the preimage $\mathsf{A}^{-1}(y)$, for a fixed $y \in \text{supp}(\rho_y)$,  is a single Dirac delta measure supported at the minimum-norm solution. Calling~\eqref{eqn:min_norm_A}, we have:
    \[
    \rho_x^\ast = \calH_\# \rho_y\,.
    \]
    This example was already shown in~\cite[Theorem 4.7] {li2023differential}.
\end{itemize}

\subsubsection{A simple nonlinear example} \label{sec:example0}
A nonlinear example that can be made explicit is as follows. Define $\calG : B_{(1,1)}(1) \to[0,1]$, where $B_{(1,1)}(1)$ is a ball centered at $(1,1)$ with radius $1$, 
\[
r = \calG(x_1,x_2) = \sqrt{{(x_1-1)^2+(x_2-1)^2}}\,.
\]
Then any preimage of $r$ is:
\[
\calG^{-1}(r) = \left\{(x_1=1+ r\cos\theta,x_2=1+ r\sin\theta)\,,\quad\theta\in[0,2\pi]\right\}\,.
\]
Through a simple calculation, within this preimage, the least-norm solution to~\eqref{eqn:least_norm_Rd} is:
\[
\calH(r) = \left(1-\frac{r}{\sqrt 2},1-\frac{r}{\sqrt 2}\right),\quad 0\leq r\leq 1\,.
\]

Suppose we are given a probability measure $\mu_r$ with a support in $[0,1]$. Then the two problems mentioned above are to find a distribution on $B_{(1,1)}(1)$ within the set of:
\[
\calS=\left\{\rho(x_1,x_2):\quad\calG_\#\rho=\mu_r\right\}\,,
\]
that minimizes entropy or the second-order moment. Explicit solutions are available:
\begin{itemize}
    \item[--]{$\calE=\iint\rho(x_1,x_2)\ln\rho(x_1,x_2) \,\rd{x_1}\rd{x_2}$:}
    \[
    \rho(x_1,x_2) = \frac{\mu_r\left(\sqrt{{(x_1-1)^2+(x_2-1)^2}}\right)}{2\pi\sqrt{(x_1-1)^2+(x_2-1)^2}}\,.
    \]
    This is a density function supported in the ball of $B_{(1,1)}(1)$, with uniform density on each ring. The density of the ring that is $r$ away from the center $(1,1)$ is scaled by $\frac{1}{2\pi r}$\footnote{Note that the integral of $\delta(\sqrt{x_1^2+x_2^2}-R)$ over the plane is $2\pi R$}. See Figure~\ref{fig:under-entropy} for an illustration of the optimal $\rho_x$.
    
    \item[--]{$\calE=\iint \left(x_1^2+x_2^2\right)\rho(x_1,x_2)\rd{x_1}\rd{x_2}$:}
    \[
    \rho(x_1,x_2) =   \begin{cases}
\sqrt{2}\mu_r(\sqrt{(x_1-1)^2+(x_2-1)^2})\delta(x_1-x_2)  \,, & 1-\frac{1}{\sqrt{2}} \leq x_1 \leq 1 \\
        0 \,, & x_1>1\,.
    \end{cases} 
    \]
    This is a density function supported only in the line section along the straight line $x_1=x_2$ between the point $(1-\frac{1}{\sqrt{2}},1-\frac{1}{\sqrt{2}})$ and the point $(1,1)$. See Figure~\ref{fig:under-moment} for an illustration of the optimal $\rho_x$ under this choice of $\calE$.
\end{itemize}

\begin{figure}[ht!]
    \centering
    \subfloat[$\calE = \int \rho_x \log \rho_x \rd x$]{\includegraphics[width=0.5\linewidth]{ 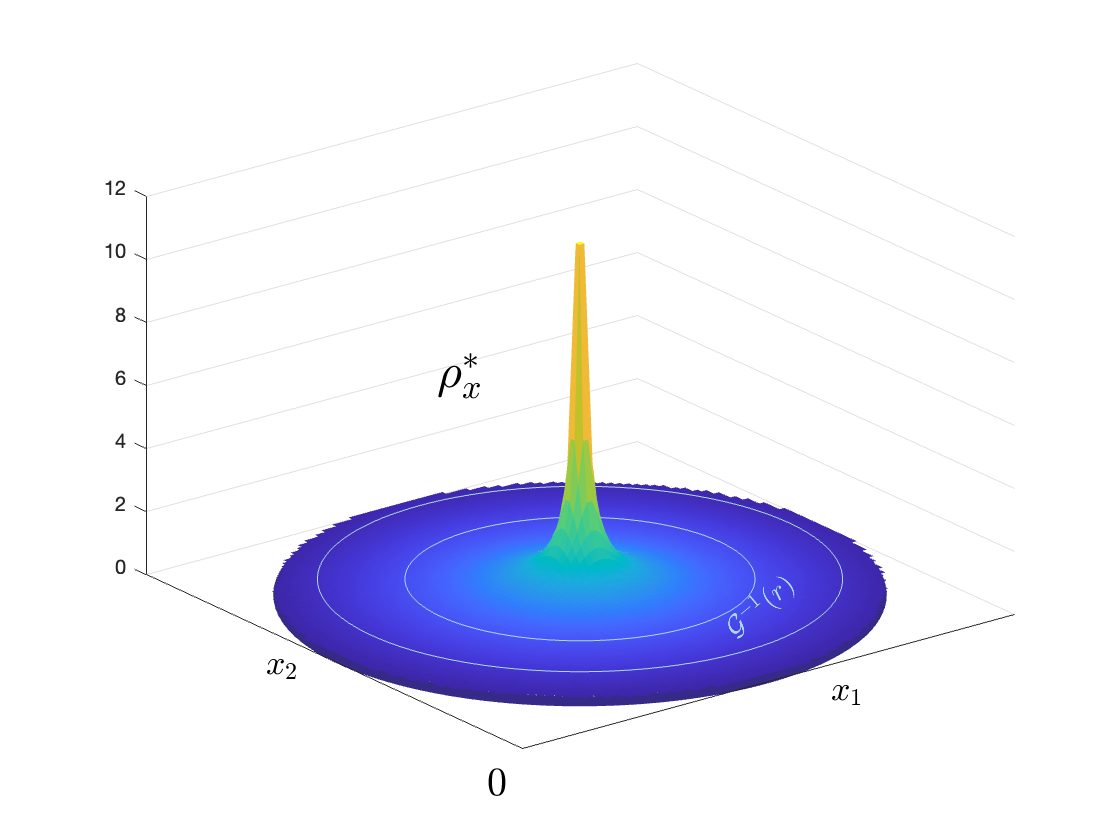}\label{fig:under-entropy}}
    \subfloat[$\calE = \int |x|^2\rho_x\rd x$]{\includegraphics[width=0.5\linewidth]{ 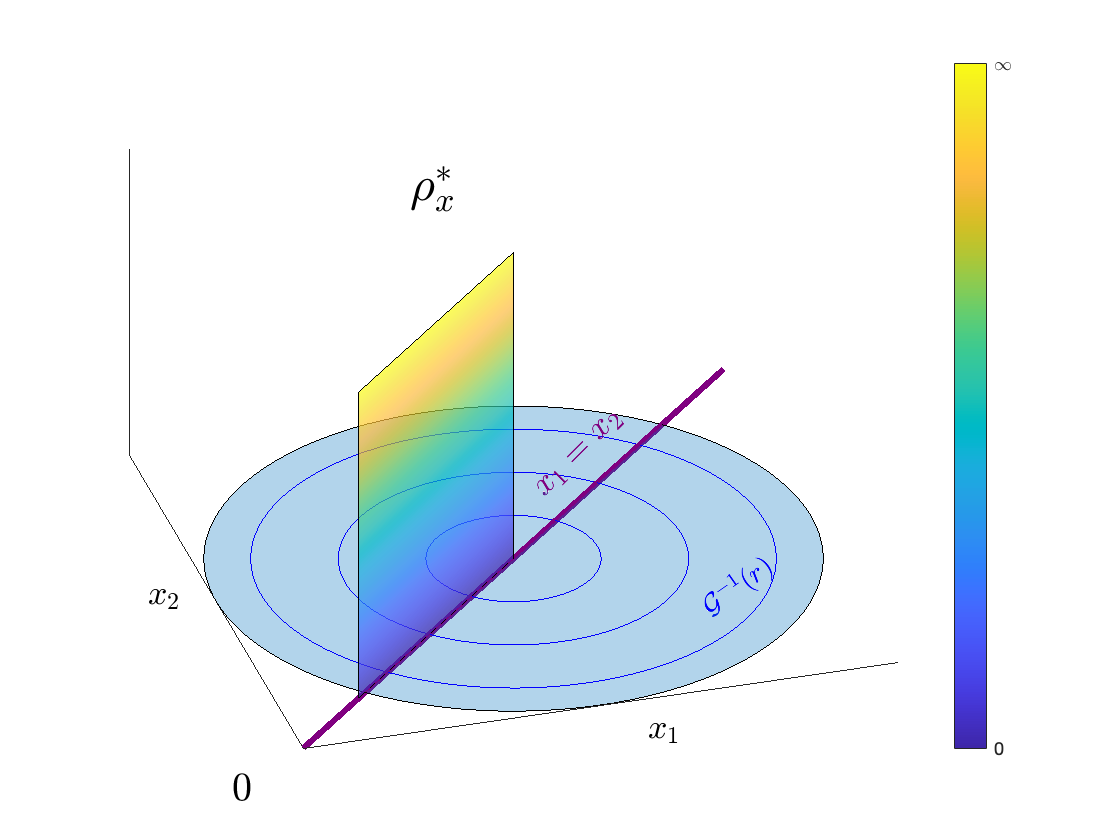}\label{fig:under-moment}}
    \caption{The data distribution  $\mu_r= \mathcal{U}([0,1])$ is the uniform distribution in $[0, 1]$.  (a): the optimizer to~\eqref{eqn:least_norm} when $\calE$ is the entropy. The density is constant when restricted to level sets of $\calG$; (b) the optimizer to~\eqref{eqn:least_norm} when $\calE$ is the second-order moment. The parameter distribution is concentrated at the element of the minimum norm on each level set of $\calG$. }
    \label{fig:under-case}
\end{figure}


\section{The Characterization of 
Solutions to problem~\eqref{eq:reg}}\label{sec:under-reg}
In this section, we characterize solutions to the regularized data-matching variational problem~\eqref{eq:reg} under two choices of the $(\calD, \mathsf{R})$, the KL-entropy pair (see Section~\ref{sec:reg-KL}) and the $W_p$-moment pair (see Section~\ref{sec:reg-Wp}).  {We illustrate some simple examples in Section~\ref{sec:example_reg} and discuss PDE-based inverse problem in Section~\ref{sec:numerics-reg}.}

\subsection{Entropy-entropy pair}\label{sec:reg-KL}
Set $\calD= \KL$ and choose $\mathsf{R}[\rho_x] = \KL(\rho_x||\mathcal{M})$, with $\calM \in \calP(\Theta)$ being a desired output probability measure for which $\frac{\rd \rho_x}{\rd \calM}$ exists. For the rest of this analysis, we assume that all probability distributions are absolutely continuous with respect to the Lebesque measure on the corresponding spaces, and we use the same notation to refer to the distribution and its corresponding density interchangeably. 
Problem~\eqref{eq:reg} can be rewritten as:
\begin{equation}\label{eq:regularized_KL}
\rho_x^*=\argmin_{\rho_x \in \mathcal P_{2,\text{ac}}}\;  \mathcal L(\rho_x)\,,\qquad \mathcal L(\rho_x): = \KL(\mathcal{G}_\#\rho_x|| \rho_y ) + \alpha  \int  \log \frac{\rho_x}{
\mathcal M}\, \rho_x \rd x \,.
\end{equation}
Under these assumptions, we have the following theorem.

\begin{theorem} \label{reg-KL}
The optimizer $\rho_x^*$ of \eqref{eq:regularized_KL} has the following property:
\begin{align} \label{rho_x_star}
\frac{\rho_x^*(\,\boldsymbol \cdot\, |{\calG^{-1}(y)}) }{\calM(\,\boldsymbol \cdot\, |{\calG^{-1}(y)}) }=   {g_\alpha}(y)
\,,\quad\forall y \in \mathcal{R}\,,
\end{align}
where $ {g_\alpha}(y)$ is a constant only depending on $y$  { and $\alpha$} while $ \rho_x^*(\,\boldsymbol \cdot\, |{\calG^{-1}(y)})$ and  $ \calM(\,\boldsymbol \cdot\, |{\calG^{-1}(y)})$  denote the conditional distributions of $\rho_x^*$ and $\calM$  on the preimage $\calG^{-1}(y)$, respectively.
In particular:
\begin{align} \label{312_revise}
\rho_x^*(x) \propto \calM (x) \left[ \frac{\rho_y(\calG(x))}{\int\delta(\calG(x)-\calG(x')) \calM(x')\rd x'} \right]^{\frac{1}{1+\alpha}}\,.
\end{align}
\end{theorem}

\begin{proof}
Since the KL divergence is convex (in the usual sense) and the pushforward action is a linear operator, the optimal solution of \eqref{eq:regularized_KL} can be obtained by solving the first-order optimality condition: 
\begin{align*}
   C_0 = \frac{\delta \mathcal L}{\delta \rho_x} \bigg|_{\rho_x^*} 
   & = 1+ \log \frac{\calG_\# \rho_x^*}{ \rho_y} (\calG(x))+ \alpha \left[ 1+ \log \frac{\rho_x^* }{\calM }(x)\right]\,,
\end{align*}
where $C_0$ is any constant. This implies that 
\begin{align} \label{317}
    (\calG_\# \rho_x^* )(\calG(x)) \left[ \frac{\rho_x^*}{\calM} (x)\right]^\alpha \varpropto \rho_y(\calG(x))\,.
\end{align}
Hence, there exists a to-be-determined function $ {g_\alpha}$ such that the following holds:
\begin{align} \label{315}
    \frac{\rho_x^*}{\calM} (x) =  {g_\alpha}(\calG(x))\,,
\end{align}
which also leads to~\eqref{rho_x_star}. 

Next, we claim that \eqref{315} leads to
\begin{align} \label{316}
    (\calG_\# \rho_x^* ) (y) =  {g_\alpha}(y) \int \delta (y- \calG(x)) \calM(x) \rd x\,.
\end{align}
Indeed, consider a mesurable function $f$ as a test function. Then by multipying both sides of \eqref{316} by $f(y)$ and integrating against $y$, we can rewrite the left-hand side (LHS) 
\begin{align*}
    \text{LHS} = \int f(y) (\calG_\# \rho_x^* ) (y) \rd y = \int f(\calG(x)) \rho_x^*(x) \rd x = \int f(\calG(x)) \calM(x)  {g_\alpha}(\calG(x)) \rd x\,,
\end{align*}
where the last equality uses \eqref{315}. For the right-hand side (RHS), we have
\begin{align*}
    \text{RHS} = \int  {g_\alpha}(y) f(y)  \int \delta (y- \calG(x)) \calM(x) \rd x \,\rd y = \int  {g_\alpha}(\calG(x)) f(\calG(x)) \calM(x) \rd x\,.
\end{align*}
Substituting \eqref{315} and \eqref{316} into \eqref{317} leads to an algebraic equation for $ {g_\alpha}$, which can be solved to obtain the final result:
\[
 {g_\alpha}(y) = \left(\frac{e^{C_0 - 1- \alpha} \rho_y(y) }{\int \delta\left(y - \mathcal{G}(x) \right) \calM(x) \rd x}\right)^{\frac{1}{1 + \alpha}}\,.
\]
Since the algebraic equation only leads to one solution, we also conclude that the function $ {g_\alpha}$ appearing in~\eqref{315} is unique.
\end{proof}
 {The prior-relative level-set structure in Theorem~\ref{reg-KL} and its dependence on $\alpha$ are tested with particles in Section~\ref{sec:numerics-reg}.}
If the regularization coefficient $\alpha = 0$, the parameter space $\Theta$ is compact, $\calM$ is the uniform distribution over $\Theta$, and $\rho_y \in \mathcal{P}(\calR)$, then Theorem~\ref{reg-KL} arrives at the same result as Theorem~\ref{thm:KL-under}.    {Alternatively, as $\alpha \to \infty$, the second term in the right hand side of \eqref{312_revise} vanishes and $\rho_x^* \to \mathcal{M} $ weakly. }  {In practice, the optimal choice of the regularization parameter $\alpha$ can be chosen to balance between the noise in the data $\rho_y$ and the strength of the prior $\mathcal{M}$; see Corollary~\ref{reg-W2} below for an explicit calculation when the map is linear, and Remark 3 in~\cite{li2024stochastic}.}

   {
There are many other possible choices of regularizers one could consider to regularize the relative entropy. However, we have not found any that yield a clean analytical formula for $\rho_x^*$. One such possibility is to study:
$$\rho_x^* = \argmin_{\rho_x \in \calP_{2, \mathrm{ac}}} \mathcal{L}(\rho_x), \quad \text{where}\quad \mathcal{L}(\rho_x) = \mathrm{KL}(\calG_\#\rho_x \|\rho_y) + \alpha \int |\rho_x(x)|^pd x.$$
Since the objective functional is strictly convex, we can characterize the minimizer $\rho_x^*$ using the first-order optimality conditions: 
\begin{equation}\label{eq:optimality}
    C_0 = \frac{\delta \mathcal{L}}{\delta \rho_x }\Big|_{\rho_x^*} = 1 + \log\frac{\calG_\#\rho_x^*}{\rho_y}(\calG(x)) + p\alpha (\rho_x^*)^{p-1}(x)
\end{equation}
which implies that $\rho_x^*$ is constant on level sets of $\calG$, i.e., $\rho_x^*(x) = g(G(x))$ for some function $g:\calR \to [0, \infty)$. Plugging back into \eqref{eq:optimality} we obtain: 
\begin{equation}\label{eq:proof1}
    C_0 = 1 + \log\frac{\calG_\#\rho_x^*}{\rho_y}(y) + p\alpha g(y)^{p - 1}. 
\end{equation}
On the other hand, by the same argument as the one presented in the proof of Theroem~\ref{thm:KL-under} we obtain that for any measurable test function $f$:
$$ \int f(y) \calG_\#\rho_x^* (y) dy = \int f(y) g(y) \int \delta (y - \calG(x)) dx \implies \calG_\#\rho_x^* (y) = g(y) \int \delta ( y - \calG(x)) dx$$
With this observation, \eqref{eq:proof1} becomes an algebraic equation for $\tilde{g}(y) := g(y)^{p-1}$, i.e.
 \begin{equation}\label{eq:lambert}
     e^{(p-1)(C_0 -1)} \left(\frac{\rho_y(y)}{\int \delta(y - \calG(x))dx}\right)^{p-1} = \tilde{g}(y) e^{\alpha p (p-1) \tilde{g}(y)},
 \end{equation} 
 Equation \eqref{eq:lambert} does not have a unique solution to $\tilde{g}$, which means there is not a unique $\rho_x^*$ to the original variational problem. Indeed, there are countably many solutions corresponding to the branches of the Lambert function~\cite{Lambert}. }
\subsection{$\wass$-moment pair}\label{sec:reg-Wp}
We now consider the case where $\calD = \mathcal{W}_p^p $, the $p$-th power of the $p$-Wasserstein metric and the regularization term  $\mathsf{R}[\rho_x]= \int |x|^p \rd \rho_x(x)$, for $1\leq p \leq \infty$.
Then problem~\eqref{eq:reg} can be rewritten as:
\begin{equation}\label{eq:regularized2}
\rho_x^\ast=\argmin_{\rho_x \in \mathcal P_p}\, \mathcal{L}(\rho_x)\,, \qquad \mathcal{L}(\rho_x) = \mathcal{W}_p^p (\mathcal{G}_\# \rho_x, \rho_y) + \alpha  \int |x|^p \rd \rho_x(x)\,.
\end{equation}

To characterize the minimizer, we first define a map $\tilde{\mathcal F}$ as follows. 
\begin{definition} \label{def:proj2}
Let $\tilde{\calF}^\alpha: \mathbb{R}^n \rightarrow \Theta\subset \mathbb{R}^m$ be defined as
\begin{equation}\label{def:reg_det}
    \tilde{\calF}^\alpha(y) = \argmin_{x\in\Theta} \left\{|\calG(x) - y|^p + \alpha |x|^p\right\}\,.
  \end{equation}
\end{definition}

Then we have the following theorem:
\begin{theorem}\label{thm:moment_regularizer}
    Given $\rho_y\in \calP(\mathbb{R}^n)$, and the forward map $\calG: \Theta \subset \mathbb{R}^m \rightarrow \calR \subset \mathbb{R}^n $ where $\calR$ and $\Theta$ are compact, then the minimizer $\rho_x^*$ to problem~\eqref{eq:regularized2} satisfies
    \begin{equation}
        \rho_x^\ast = \tilde{\calF}^\alpha_\# \rho_y\,.
    \end{equation}
\end{theorem}

   {We again observe a contrast between the Wasserstein-moment and entropy-entropy pairings. Results derived from the Wasserstein-moment pairing preserve the particle-level Euclidean structure and call for the map $\tilde{\calF}^\alpha$, mirroring their deterministic counterparts, whereas the entropy-entropy pairing generally lacks such a direct pointwise analog.}

The proof of Theorem \ref{thm:moment_regularizer} is built upon a nice observation about the regularizer:
\[
\mathsf{R}[\rho_x]=\int |x|^p \rd \rho_x(x)=\mathcal{W}_p^p(\rho_x,\delta_0)\,.
\]
which allows us to rewrite the loss function as follows:
\begin{lemma}\label{lem:was_reg_1}
For any $\rho_y \in \calP(\mathbb{R}^n)$, the objective function defined in~\eqref{eq:regularized2} can be rewritten as:
\begin{align} \label{W2-reg}
\mathcal L(\rho_x)=\mathcal{W}_p^p (\mathcal G_\# \rho_x, \rho_y) + \alpha \int |x|^p \rd \rho_x(x) = \mathcal{W}_p^p (\tilde \calG_\# \rho_x, \bar \rho_y) \,,  
\end{align}
with $\bar \rho_y = \rho_y \otimes \delta_0(x)$ where $\delta_0(x) \in \calP(\mathbb{R}^{m})$ denotes the Dirac delta centered at $0\in \mathbb{R}^{m}$, and $\tilde{\mathcal G} =      \mathcal G  \otimes \left( \alpha^{1/p}\,  \mathsf{I}_m\right) $, with $\mathsf{I}_m$ being the $m$-dimensional identity matrix. More explicitly,
\[
\tilde{\calG}(x):\Theta\subset\mathbb{R}^m\to\calR \times \Theta\subset\mathbb{R}^{n+m}\,,\quad\text{with}\quad \tilde{\calG}(x)= \begin{pmatrix}
    \calG(x)\\
    \alpha^{1/p} x
\end{pmatrix}\,.
\]
\end{lemma}

We will leave the proof for this lemma later. Applying this Lemma~\ref{lem:was_reg_1}, we obtain that problem~\eqref{eq:regularized2} is equivalent to
\begin{equation}
    \rho_x^\ast = \argmin_{\rho_x \in \mathcal{P}_p} \wass  (\tilde{\calG}_\# \rho_x, \bar{\rho}_y)\,.
\end{equation}
This can be considered as an  overdetermined problem in a similar form with~\eqref{eq:W_min} since 
$$
\tilde{\calG}: \Theta \subset \mathbb{R}^m \rightarrow \calR \times \Theta \subset \mathbb{R}^{n+m}\,,
$$
for which the range lies in a strictly higher-dimensional space than its domain. As a result, Theorem~\ref{thm:marginal} directly applies, which allows us to write down an explicit form of the solution to problem~\eqref{eq:regularized2} as below.

\begin{proof}[Proof for Theorem~\ref{thm:moment_regularizer}]
In this overdetermined system, according to Theorem~\ref{thm:marginal}, we simply need to find the function that solves:
\[
    \calF^{\ex}(\mathbf{y}^*) = \argmin_{x \in \Theta} |\tilde{\calG}(x)-\mathbf{y}^*|^p\,,\quad \text{for}\,\, \mathbf{y}^*=(y,0)\in\mathbb{R}^{n+m}\,.
\]
Then the solution is $\rho_x^\ast=\calF^{\ex}_{\#}\bar{\rho}_y$. Comparing with Definition~\ref{def:proj2}, it is straightforward to see that $\calF^{\ex}((y,0))=\tilde\calF^\alpha(y)$. Given the specific form of $\bar\rho_y$, we conclude $\rho_x^\ast=\tilde\calF_{\#}^\alpha\rho_y$.
\end{proof}


\begin{proof}[Proof for Lemma~\ref{lem:was_reg_1}]
We drop the sub-index $m$  on the identity matrix because there is no ambiguity.  Let $\pi_1$ be the optimal transport plan between $\calG_\# \rho_x$ and $\rho_y$ in the optimal transport problem defining the $\mathcal{W}_p$ metric. Then
\[
\mathcal{W}_p^p(\mathcal G_\# \rho_x, \rho_y)  = \int |y' - y|^p \pi_1 (\rd y' \rd y) = \int |\mathcal G(x) - y|^p \hat{\pi}_1 (\rd x \rd y),
\]
where ${\pi}_1  = (\calG \times \mathsf{I})_\#  \hat {\pi}_1$ for some $\hat{\pi}_1 \in  \Gamma(\rho_x, \rho_y)$ and $\Gamma(\rho_x, \rho_y)$ denotes all the couplings between $\rho_x$ and $\rho_y$. Note that if $\calG$ is not one-to-one, $\hat{\pi}_1$ may not be unique, but its existence is always guaranteed. Similarly:
\begin{equation}
    \int |x|^p \rd\rho_x=\int |x-0|^p \hat{\pi}_2(\rd x \rd x')\,,\quad\text{with}\quad \hat{\pi}_2 = \rho_x  \otimes \delta_0(x) \in \Gamma(\rho_x, \delta_0(x)) \,,
\end{equation}
where $\Gamma(\rho_x, \delta_0(x))$ denotes all the couplings between $\rho_x$ and $\delta_0(x)$, and $\delta_0(x) \in \calP(\mathbb{R}^m)$ denotes the Dirac delta at $0\in \mathbb{R}^m$. 
Defining $\hat{\pi}_3  =\hat{\pi}_1\otimes \delta_0(x) \in \Gamma( \rho_x, \rho_y  \otimes \delta_0(x))$, we rewrite:
\begin{align*}
\mathcal{L}(\rho_x) =  & \int |\mathcal G(x) - y|^p \, \hat{\pi}_1 (\rd x \rd y) + \alpha \int |x|^p \rd\rho_x 
    \\  = & \int |\tilde{\mathcal G} (x) - {\bf y}'|^p \,\hat{\pi}_3( \rd x\, \rd {\bf y}')\quad\text{with} \quad {\bf y}'=(y,0)\\
    = &   \int |{\bf y} -{\bf y}'|^p\, {\pi}_3 (\rd {\bf y}  \, \rd{\bf y}' ) \,, \quad {\pi}_3 = (\tilde{\mathcal{G}} \times \mathsf{I})_\#  \hat {\pi}_3  \in \Gamma\left( \tilde{\mathcal{G}}_\# \rho_x, \, \rho_y  \otimes \delta_0(x) \right)\,.
    \end{align*}

To show this is $\mathcal{W}_p^p (\tilde{\calG}_\#\rho_x,\bar\rho_y)$, we also need to show $\pi_3$ is an optimal plan. Assume $\gamma\neq \pi_3$ and $\gamma$ is the optimal transport plan between $ \tilde{\mathcal{G}}_\# \rho_x$ and  $\bar{\rho}_y = \rho_y \otimes \delta_0(x)$ under the cost function $c(\mathbf{y}-\mathbf{y}') = |\mathbf{y}-\mathbf{y}'|^p$, then we have 
\begin{align*}
    \mathcal{W}_p^p({\tilde \calG}_\# \rho_x, \bar \rho_y) = &  \int |{\bf y} -{\bf y}'|^p  \rd\gamma (\rd {\bf y}  \, \rd{\bf y}' ) \\
    = & \int |\tilde{\mathcal G} (x) - {\bf y}'|^p  \rd\hat{\gamma}( \rd x\, \rd {\bf y}'),\quad \gamma = (\tilde{\mathcal{G}} \times \mathsf{I})_\# \hat{\gamma},\quad \hat{\gamma}\in \Gamma\left(\rho_x, 
       \rho_y  \otimes \delta_0(x) \right) \\
    = & \int |\mathcal G(x) - y|^p  \rd \hat{\gamma}_1 ( \rd x \,  \rd y) + \alpha \int |x|^p  \rd \rho_x  \,, \quad \hat{\gamma}_1 \in \Gamma(\rho_x, 
        \rho_y ) \\
    = & \int |y - y'|^p  \rd \hat{\gamma}_2 (\rd y \, \rd y') + \alpha \int |x|^p \rd \rho_x  \,, \quad  \hat{\gamma}_2  = (\calG \times \mathsf{I})_\# \hat{\gamma}_1  \in \Gamma(\calG_\# \rho_x, 
        \rho_y ) \\
    \geq & \,\, \mathcal{W}_p^p (\mathcal G_\# \rho_x, \rho_y ) + \alpha \int |x|^p \rd \rho_x\,, \\
    = & \int |{\bf y} -{\bf y}'|^p {\pi}_3 (\rd {\bf y}  \, \rd{\bf y}' ) \,,
\end{align*}
where $ \hat{\gamma}_1$ and $\hat{\gamma}_2$ are determined by $\gamma$. This
contradicts the assumption that $\pi_3$ is not optimal. So we conclude with~\eqref{W2-reg}.
\end{proof}

   {Note that as $\alpha \to 0$, $\tilde{\calF}^\alpha$ $\Gamma$-converges to  $\tilde{\cal F} (y) = \argmin_{x\in \Theta} |\calG(x) - y|^p$, and hence the solution $\tilde{\calF}^\alpha _\#\rho_y \to \tilde{\calF}_\#\rho_y$. On the other hand, as $\alpha \to \infty$, $\tilde{\calF}^\alpha \to x^* := \argmin_{x \in \Theta}|x|^p$, the minimum norm element in $\Theta$, which implies that $\rho_x^* \to \delta_{x^*}$. }

\subsection{Examples}\label{sec:example_reg}
Similar to what is done in Sections~\ref{sec:example_over} and~\ref{sec:example_under}, we design a linear and a simple nonlinear example for which we can explicitly spell out the solutions to the regularized problem.
\subsubsection{A linear pushforward map}

We first assume $\calG=\mathsf{A}\in\mathbb{R}^{m\times n}$ with $m\geq n$. Setting $\calD = \mathcal{W}_2^2 $ and $\mathsf{R}[\rho_x]= \int |x|^2 \rd \rho_x(x)$, the solution to~\eqref{eq:reg}, according to Theorem~\ref{thm:moment_regularizer}, is:
    \begin{align} \label{rhous}
\rho_x = \left( (\mathsf{A}^\top \mathsf{A} + \alpha \mathsf{I})^{-1} \mathsf{A}^\top\right)_\# \rho_y  \,.
\end{align}
This can be easily deduced from Definition~\ref{def:proj2} with $\calG$ replaced by the linear operator.

One fascinating fact is that the Wasserstein-moment setting resonates nicely with Tikhonov regularization. In the Euclidean space, Tikhonov regularization was introduced to smooth the signifying effect of small singular values in $\mathsf{A}$ on small perturbations in the data. Namely, suppose the ground truth data $y$ is perturbed to become $y+\delta$ as the given data. In that case, one needs to regularize the problem to temper the amplification of the smallest singular value of $\mathsf{A}$. The choice of $\alpha$ thus plays an important role in controlling the error magnification. The same phenomenon is observed in the probability space as well, which is stated in the Corollary~\ref{reg-W2} below.

\begin{corollary} \label{reg-W2}
Denote $\rho_y^\true$ the groundtruth probability measure and $\rho_y^\delta$ its $\delta$-perturbation. They are both in $\calP_2(\mathbb{R}^n)$. Denote then $\rho_x^\true$ the solution to problem~\eqref{eq:W_min} with the reference data distribution $\rho_y^\true$, and $\rho_x^{\delta,\alpha}$ the solution to problem~\eqref{eq:regularized2} using regularizing coefficient $\alpha$ on the reference data distribution $\rho_y^\delta$. Then we have that
\begin{align} \label{12}
\mathcal{W}_2(\rho_x^\true , \rho_x^{\delta,\alpha}) \leq  \|(\mathsf{A}^\top \mathsf{A} + \alpha \mathsf{I})^{-1} \mathsf{A}^\top\|_2 \mathcal{W}_2(\rho_y^\true, \rho_y^\delta ) 
+
\|(\mathsf{A}^\top \mathsf{A} + \alpha \mathsf{I})^{-1} \mathsf{A}^\top - \mathsf{A}^{\dagger}\|_2 \sqrt{{\mathbb{E}_{\rho_y^\true} \left[y^2\right]}}\,. 
\end{align}
Set $\sigma_m=\sigma_{\min}(\mathsf{A})$ as the smallest singular value for $\mathsf{A}$ and denote by $\rho(\mathsf{A})$ the set of all singular values of $\mathsf{A}$. The bound in~\eqref{12} can be further simplified to 
\begin{align}
\mathcal{W}_2(\rho_x^\true, \rho_x^{\delta,\alpha}) 
& \leq \left(  \max_{\sigma_i \in \rho(\mathsf{A})} \frac{\sigma_i}{\alpha + \sigma_i^2} \right) \mathcal{W}_2(\rho^\true_y, \rho^\delta_y) +  \frac{\alpha}{\sigma_m(\sigma_m^2+\alpha)} \sqrt{\mathbb{E}_{\rho^\true_y} \left[|y|^2\right]}  \nonumber \\
&\leq  \frac{1}{2\sqrt{\alpha}} \mathcal{W}_2(\rho^\true_y, \rho^\delta_y) +   \frac{\alpha}{\sigma_m(\sigma_m^2+\alpha)}\sqrt{\mathbb{E}_{\rho^\true_y} \left[|y|^2\right]} \label{eqn:12_simplify_2} \,.
\end{align}
\end{corollary}
This corollary clearly shows that the error (as written in \eqref{eqn:12_simplify_2}) has two sources: the former being the noise in the measurement and the latter coming from the regularization. It is common to equate these two contributions to seek the optimal choice of $\alpha$ that minimizes the combination of two errors when $\sigma_m$ is known.
\begin{proof}
To show~\eqref{12}, we deploy the triangle inequality:
\begin{equation*}
 \mathcal{W}_2 (\rho_x^\true , \rho_x^{\delta,\alpha}) \leq
\mathcal{W}_2(\tilde \rho_x, \rho_x^{\delta,\alpha}) + \mathcal{W}_2(\rho_x^\true , \tilde \rho_x)\,,
\end{equation*}
where
\begin{equation}\label{eqn:middle_term}
\tilde \rho_x=\left( (\mathsf{A}^\top \mathsf{A} + \alpha \mathsf{I})^{-1} \mathsf{A}^\top\right)_\# \rho_y^\true\,.
\end{equation}
Notice in this setting, by~\eqref{rhous}, $\rho_x^{\delta,\alpha}=\left(  (\mathsf{A}^\top \mathsf{A} + \alpha \mathsf{I})^{-1} \mathsf{A}^\top\right)_\# \rho_y^{\delta}$. By utilizing the continuity of the map $(\mathsf{A}^\top \mathsf{A} + \alpha \mathsf{I})^{-1} \mathsf{A}^\top$ and citing~\cite[Theorem 3.2]{ernst2022wasserstein}, we obtain the first term in~\eqref{12}. To control the second term, we expand:
    \begin{eqnarray*}
        \mathcal{W}_2^2 (\tilde \rho_x, \rho_x^\true) &=& \mathcal{W}_2^2\left(\left((\mathsf{A}^\top \mathsf{A} + {\alpha} \mathsf{I})^{-1} \mathsf{A}^\top\right)_\# \rho^\true_y\,,\, \mathsf{A}^{\dagger}_\# \rho^\true_y\right) \\
        &\leq & \int \left| \left(\mathsf{A}^\top \mathsf{A} + {\alpha} \mathsf{I}\right)^{-1} \mathsf{A}^\top y -\mathsf{A}^{\dagger} y \right|^2 \rd \rho^\true_y(y) \\
        &\leq & \left\|\left(\mathsf{A}^\top \mathsf{A} + {\alpha} \mathsf{I}\right)^{-1} \mathsf{A}^\top -\mathsf{A}^{\dagger}\right\|^2_2 \underbrace{\int |y|^2\rd\rho^\true_y(y)}_{={\mathbb{E}_{\rho^\true_y} \left[|y|^2\right]}}\,
    \end{eqnarray*}
The same estimates was seen in \cite[Theorem 3.1]{baptista2023approximation}. From~\eqref{12}, we bound matrices' $2$-norm using:
\begin{align*}
\mathcal{W}_2(\rho_x^\true, \rho_x^{\delta,\alpha}) \leq \left(  \max_{\sigma_i \in \rho(\mathsf{A})} \frac{\sigma_i}{\alpha + \sigma_i^2} \right) \mathcal{W}_2(\rho^\true_y, \rho^\delta_y) +  \frac{\alpha}{\sigma_m(\sigma_m^2+\alpha)} \sqrt{\mathbb{E}_{\rho^\true_y} \left[|y|^2\right]}\,.
\end{align*}
Finally, we arrive at the conclusion by noting that $\max\limits_{\sigma_i \in \rho(\mathsf{A})} \frac{\sigma_i}{\alpha + \sigma_i^2}\leq \frac{1}{2\sqrt{\alpha}}$.
\end{proof}

\subsubsection{A simple nonlinear example}
In this section, we revisit the example in Section~\ref{sec:example0} but consider the optimization formulation of~\eqref{eq:reg}.  
Recall $\calG : B_{(1,1)}(1) \to[0,1]$, where $B_{(1,1)}(1)$ is a ball centered at $(1,1)$ with radius $1$, 
\[
r = \calG(x_1,x_2) = \sqrt{{(x_1-1)^2+(x_2-1)^2}}\,.
\]
Then any preimage of $r$ is:
\[
\calG^{-1}(r) = \left\{(x_1=1+ r\cos\theta,x_2=1+ r\sin\theta)\,,\quad\theta\in[0,2\pi]\right\}\,.
\]

Suppose we are given a probability measure $\mu_r$ with a support in $[0,1]$. Then the two problems mentioned above are to find a distribution on $B_{(1,1)}(1)$ that minimmizes~\eqref{eq:reg} under different choices of $\mathsf{R}$. Explicit solutions are available:
\begin{itemize}
    \item[--]{$\calD= \KL$ and $\mathsf{R}[\rho_x] = \KL(\rho_x||\mathcal{M})$}. This is the formulation deployed in~\eqref{eq:regularized_KL}. According to Theorem~\ref{reg-KL}, especially Equation~\eqref{312_revise}, when $\alpha = 1$:
    \[
    \rho^*(x_1,x_2) \propto   \calM (x_1,x_2) \sqrt{ \frac{\mu_r\left(\calG(x_1,x_2)\right)}{\int \delta(\calG(x_1,x_2)-\calG(x_1',x_2')))\mathcal M(x_1',x_2') \rd{x_1'}\rd{x_2'}} }\,.
    \]
    In this context, we need to compute the denominator term 
    $$
    \int \delta(\calG(x_1,x_2)-\calG(x_1',x_2')))\mathcal M(x_1',x_2') \rd{x_1'}\rd{x_2'}\,.
    $$ 
    The analytic solution is not available for an arbitrarily given $\mathcal M$ but is explicit when $\calM$ has special forms. One such example is when $\calM = C_\calM\exp\left(-\frac{x_1^2+x_2^2}{2}\right)$ where $C_\calM$ is the normalizing constant. Under this setup, we define $(x_1',x_2'):=(1+r'\cos\theta,1+r'\sin\theta)$ for a given $r'$ and $\theta$, and set $r=\calG(x_1,x_2)$. Upon conducting a change of variable:
\begin{align*}
       & \quad \int \delta(\calG(x_1,x_2)-\calG(x_1',x_2')))\mathcal M(x_1',x_2') \rd{x_1'}\rd{x_2'}\\
       &=\int \delta(r - r') \calM(1+r'\cos\theta,1+r'\sin\theta) \, r' \, \rd r' \, \rd\theta \\
    &= r C_\calM \int_0^{2\pi} \exp\left( -\frac{1}{2}\left((r\cos\theta+1)^2 + (r\sin\theta+1)^2\right) \right) \rd\theta \\
    &=  r C_\calM  \exp\left(-\frac{1}{2}(2+r^2)\right) \int_0^{2\pi} \exp\left(-r \cos\theta - r \sin\theta \right)\rd\theta\\
    &= 2\pi r C_\calM \exp\left(-\frac{1}{2}(2+r^2)\right) I_0(\sqrt{2}r) \,,
    \end{align*}
where $I_0$ is the Bessel function\footnote{Bessel function is defined as $I_0(a) = \frac{1}{2\pi}\int_0^{2\pi}\exp\left( a\cos\theta\right)\rd\theta$.}. This leads to the final result of
\[
\rho^*(x_1,x_2) \propto   \exp\left(-\frac{x_1^2+x^2_2}{2}\right)\sqrt{\frac{\mu_r(r)}{2\pi re^{-\frac{1}{2}(2+r^2)}I_0(\sqrt{2}r)}}\,,\quad\text{where}\quad r=\calG(x_1,x_2)\,.
\]

We present the optimal solution, $\rho^*$, in Figure~\ref{fig:reg_KL} for both the unregularized case $(\alpha = 0)$ (also depicted in Figure~\ref{fig:under-entropy}) and the regularized case $(\alpha = 1)$. Given that the prior distribution $\mathcal{M}$ is centered at $x = 0$, the regularized solution (Figure~\ref{fig:reg_KL_1}) combines characteristics of both the prior distribution $\mathcal{M}$ and the unregularized distribution shown in Figure~\ref{fig:reg_KL_0}.
\begin{figure}[ht!]
    \centering
    \subfloat[$\alpha = 0$]{\includegraphics[width=0.5\linewidth]{ 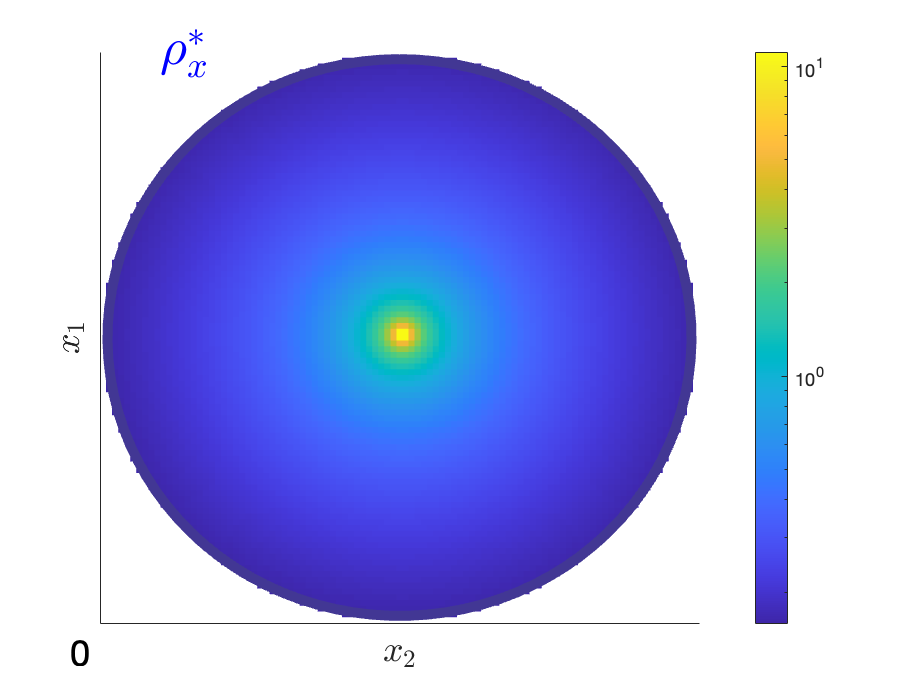}\label{fig:reg_KL_0}}
    \subfloat[$\alpha = 1$]{\includegraphics[width=0.5\linewidth]{ 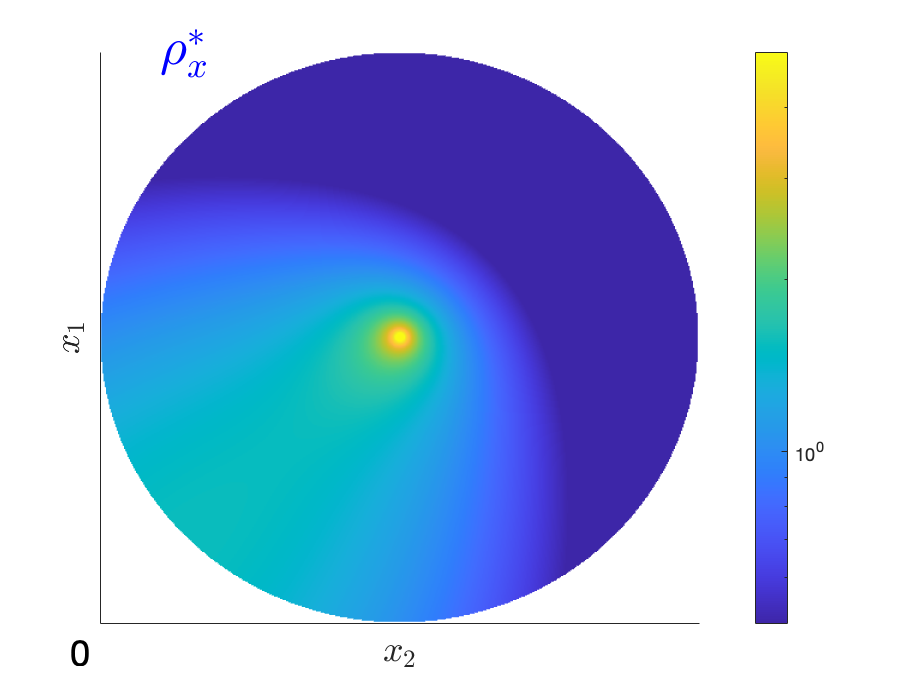}\label{fig:reg_KL_1}}
    \caption{A comparison of the reconstructed density for (a) the unregularized case with entropy cost $\calE[\rho_x] = \int\rho_x \ln \rho_x \rd x$ and (b)~the regularized case with  $(\calD, \mathsf{R}) = (\text{KL},\text{KL}(\cdot \|\calM))$ where $\calM \propto \exp\left(-\frac{x_1^2 + x_2^2}{2}\right)$ and the parameter $\alpha = 1$. The data distribution is chosen to be $\mu_r = \mathcal{U}([0, 1])$, the uniform distribution in $[0, 1]$}
    \label{fig:reg_KL}
\end{figure}

    \item[--]
    $\calD = \mathcal{W}_2^2 $ and $\mathsf{R}[\rho_x]= \int |x|^2 \rd \rho_x(x)$. Then according to Definition~\ref{def:proj2}, we have, in this specific context:
    \[
    \mathcal{F}(r) = \argmin_{(x_1,x_2)\in B_{(1,1)}(1)} \left\{|\calG(x_1,x_2) - r|^2 + \alpha |x_1|^2+\alpha|x_2|^2\right\}\,.
    \]
    Through simple calculation, it can be shown that the optimal solution has to be supported along the diagonal set $\{(x_1,x_2)\in B_{(1,1)}(1): x_1 = x_2\}$, reducing the problem to solving:
  \[
  \begin{aligned}
F(r)&:= \argmin_{x \in [0,1]} (\sqrt{2}|x-1|-r)^2 + 2\alpha x^2\\
 &=\argmin_{x \in [0,1]} \left[(2+2\alpha)x^2+(-4+2\sqrt{2}r)x+(2-2\sqrt{2}r+r^2)\right]\,.
 \end{aligned}
  \]
  Since this objective function is quadratic, the solution is explicit:
  \[
  \mathcal{F}(r)=(F(r),F(r))\,,\quad\text{with}\quad F(r) = \frac{1-\frac{\sqrt{2}}{2}r}{1+\alpha}\,.
  \]
  According to Theorem~\ref{thm:moment_regularizer}, we have $\rho_x^\ast = \mathcal{F}_\# \mu_r$. It is a distribution concentrated along the diagonal.

  In Figure~\ref{fig:reg_Wp0}, we revisit the unregularized solution where $\mathcal{E}[\rho_x] = \int |x|^2 \, \mathrm{d}\rho_x$ (also shown in Figure~\ref{fig:under-moment}). In Figure~\ref{fig:reg_Wp1}, we present the optimal parameter solution with regularization. The dotted circles in Figure~\ref{fig:reg_Wp0} highlights the level sets $\mathcal{G}^{-1}(r)$ for a fixed $r$. Only a single point closest to the origin on the level set is selected according to the function $\mathcal{H}$ (see Section~\ref{sec:example0}). 

With a nonzero regularization coefficient, the map shifts from $\mathcal{H}$ to its affine transform $\mathcal{F}$, which depends on the coefficient $\alpha$. As a result, the original level sets (dotted circles in Figure~\ref{fig:reg_Wp1}) are translated into new level sets (solid regions in Figure~\ref{fig:reg_Wp1}), and the point closest to the origin is selected to inherit all the mass from the data distribution at $r$. In other words, the two distributions in Figure~\ref{fig:reg_Wp} are related through a linear pushforward map.
\end{itemize}

\begin{figure}[ht!]
     \centering
    \subfloat[$\alpha = 0$]{\includegraphics[width=0.5\linewidth]{ 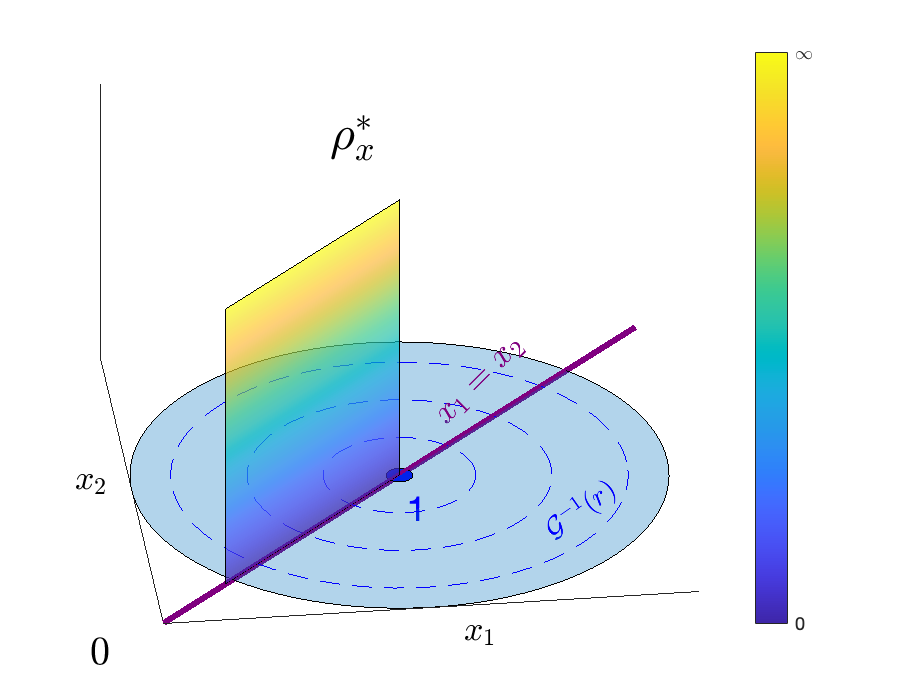}\label{fig:reg_Wp0}}
    \subfloat[$\alpha \neq 0 $]{\includegraphics[width=0.5\linewidth]{ 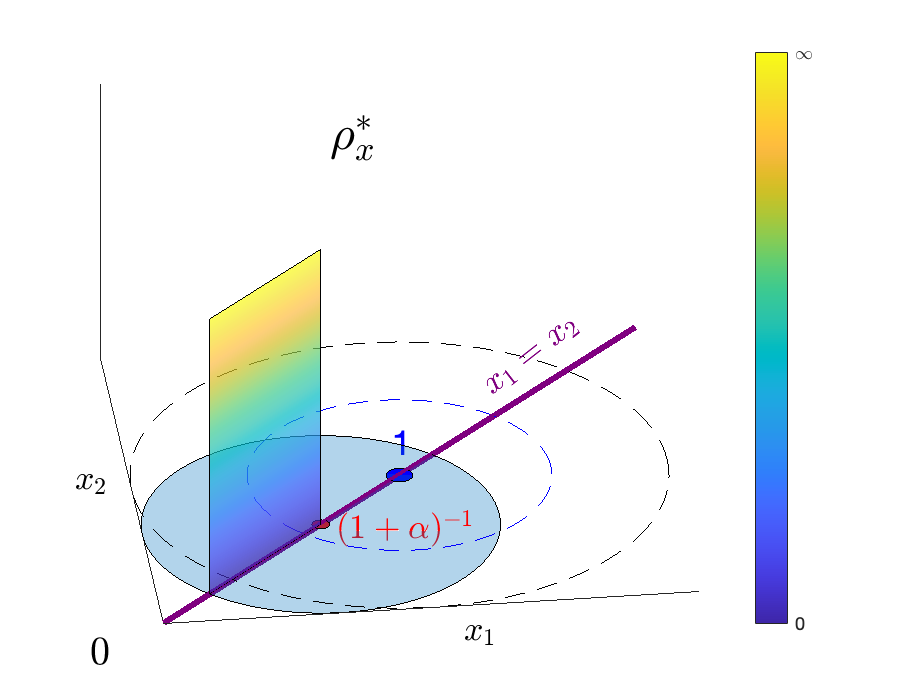}\label{fig:reg_Wp1}}
    \caption{A comparison of the optimal distributions under different setups: (a) shows the result without regularization and the minimum norm cost $\calE[\rho_x] = \int |x|^2 \rd \rho_x$, while (b) shows the result for the regularized version with $(\calD,\mathsf{R}) = (\mathcal{W}_2^2, \int |x|^2 \rd \rho_x(x))$ and  $\alpha\neq 0$.} 
    \label{fig:reg_Wp}
\end{figure}

{  
\section{Numerical Experiments for a PDE Inverse Problem}\label{sec:numerics}
We apply the formulations \eqref{eqn:variation}, \eqref{eqn:least_norm}, \eqref{eq:reg} to a PDE-based inverse problem in this section. Since the solution has nonlinear dependence to the parameters, the pushforward map $\mathcal{G}$ is not explicitly available, and we will deploy a gradient based solver to find the numerical optimizer. Below we first discuss the problem setup and our numerical strategy in Section~\ref{sec:num:preparation} and demonstrate numerical results for over-, under- and regularized cases respectively in the following subsections.

\subsection{Setup: A two-dimensional elliptic PDE inverse problem}\label{sec:num:preparation}
Let $D=(0,1)^2$ and write $z=(z_1,z_2)$ for the spatial variable. We have a two-dimensional elliptic equation:
\begin{equation}\label{eq:pde-over}
\begin{cases}
-\nabla_z\!\cdot\!\bigl(A_\theta\nabla_z u_\theta\bigr)
=S & \text{in }D,\\
u_\theta=0 & \text{on }\partial D,
\end{cases}
\end{equation}
where the source term $S(z)=\sum_{kl}\varphi_{kl}(z)$ with each basis $\varphi_{kl}(z)=\sin(k\pi z_1)\sin(l\pi z_2)$. The unknown parameters are in the coefficient $A_\theta$.

We enumerate the two scenarios in Table~\ref{tab:two_pdes}. In the table, we define the normalized modal observation
\[
\mathcal O_{kl}(u)=
\frac{\langle u,\varphi_{kl}\rangle_{D}}
{\langle\varphi_{kl},\varphi_{kl}\rangle_{D}}\,.
\]
\begin{table}[h!]
    \centering
    \begin{tabular}{|c|c|c|}
    \hline
    scenario & parameter $A_{\boldsymbol\theta}$ & measurements\\
    \hline
        Case 1 & $\operatorname{diag}(1+\theta,1)$ with $\boldsymbol\theta=\theta\in[0,1]$ & $\calG_{\rm over}(\theta)
=\bigl(\mathcal O_{11}(u_\theta),\mathcal O_{21}(u_\theta)\bigr)$ \\
\hline
        Case 2 & $\operatorname{diag}(1+\theta_1,1+\theta_2)$ with $\boldsymbol\theta=(\theta_1,\theta_2)\in[0,1]^2$ & $\calG_{\rm under}(\boldsymbol\theta)
=\mathcal O_{11}(u_{\boldsymbol\theta})$\\
\hline
    \end{tabular}
    \caption{Two simulation cases for the 2D elliptic PDE inverse problem.}
    \label{tab:two_pdes}
\end{table}

Clearly, in Case 1, a one-dimensional parameter is mapped to a two-dimensional observation quantity and thus is overdetermined, while Case 2 reverses the situation and is underdetermined. Due to the simplicity of the PDE form, in this specific case one has explicit form for $\mathcal{G}$:
\begin{equation}\label{eq:pde-over-map}
y(\theta)=\calG_{\rm over}(\theta)=\bigl(\mathcal O_{11}(u_\theta),\mathcal O_{21}(u_\theta)\bigr)
=\left(\frac{1}{\pi^2(2+\theta)},
\frac{1}{\pi^2(5+4\theta)}\right).
\end{equation}
and
\begin{equation}\label{eq:pde-under-map}
y(\boldsymbol\theta)=\calG_{\rm under}(\boldsymbol\theta)
=\mathcal O_{11}(u_{\boldsymbol\theta})
=\frac{1}{\pi^2(2+\theta_1+\theta_2)}.
\end{equation}
However, for presentation purpose, we pretend the explicit $\mathcal{G}$ is unknown and simulate PDE using a centered five-point finite-difference stencil. Specifically, we discretize $D=(0,1)^2$ on a uniform grid with $n=63$ interior points in each direction and mesh width $h=(n+1)^{-1}$ and the inner product in the observations are simulated using vector products. We use $\mathcal G_h$ to denote the numerical approximation of the PDE forward map.

To numerically demonstrate all six characterizations developed in Sections~\ref{sec:over}--\ref{sec:under-reg} on this elliptic PDE models, we will use the particle method with an empirical measure of $N=2000$ particles as an approximation. Denoting by $k$ the iteration in optimization, the approximate numerical optimization solution writes as:
\[
\rho_N^k=\frac1N\sum_{i=1}^N\delta_{\boldsymbol\theta_i^k}.
\]
None of the theorem-predicted optimizer information is supplied to the solver. We will observe the optimization solutions and study their characterizations.


\subsection{Overdetermined reconstruction}\label{sec:numerics-over}
We choose Case 1 for our simulation to demonstrate the overdetermined reconstruction. We use one parameter to parameterize $A_\theta$, and measure two observations $\calG_{\rm over}$ as in~\eqref{eq:pde-over-map}.  {The data consist of $N=2000$ observations: $85\%$ are PDE outputs generated with $\theta\sim\rho_\theta^\dagger$, and $15\%$ are drawn from a three-component off-range Gaussian contamination law.  The ground-truth parameter density, its clean pushforward, and the resulting observed distribution are shown in Figure~\ref{fig:pde-over-data-generation}.}

\begin{figure}[ht!]
    \centering
    \includegraphics[width=0.315\textwidth]{ 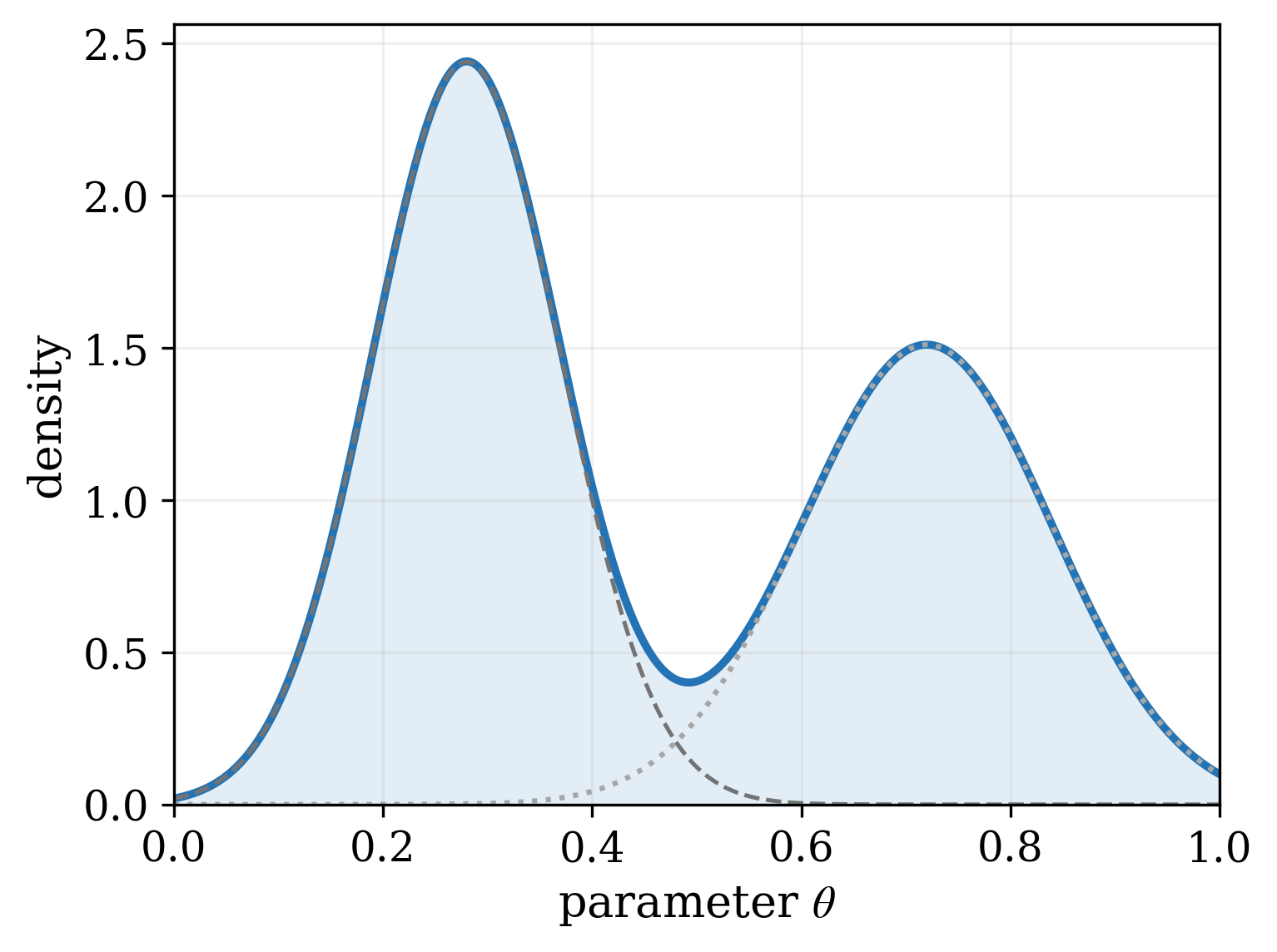}
    \hfill
    \includegraphics[width=0.315\textwidth]{ 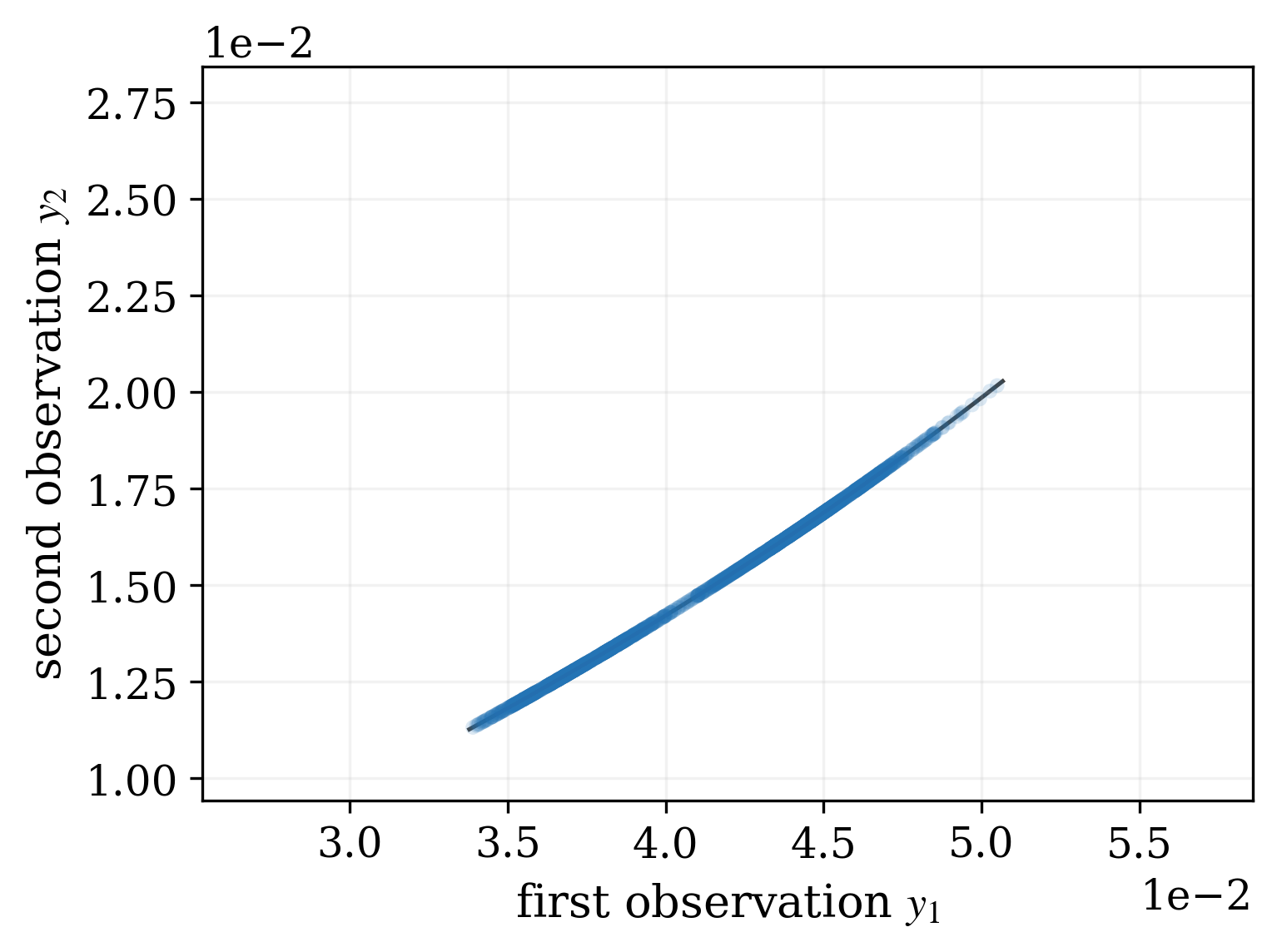}
    \hfill
    \includegraphics[width=0.315\textwidth]{ 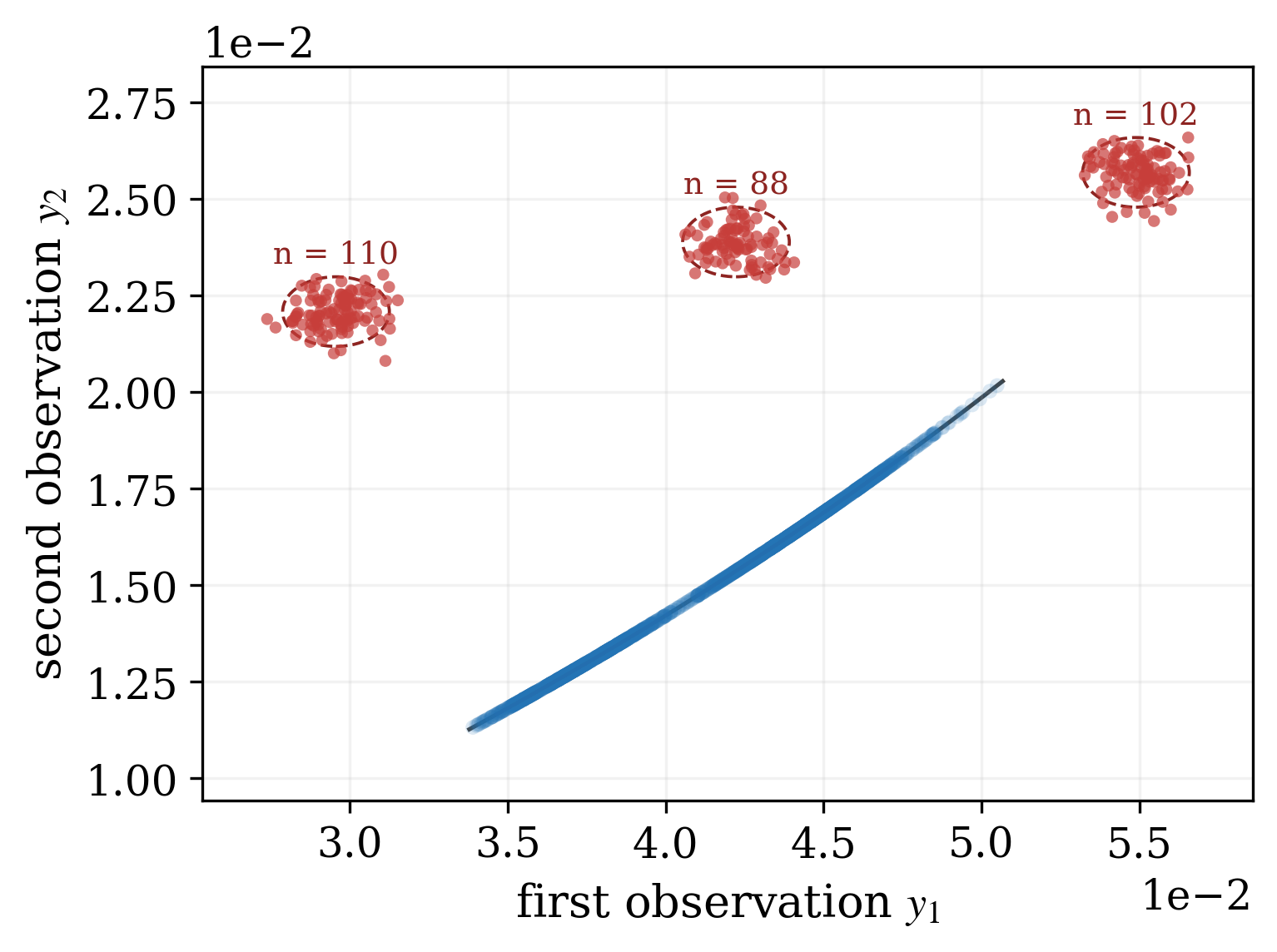}
    \caption{ {Data generation for the overdetermined experiment.  Left: the bimodal ground-truth parameter density $\rho_\theta^\dagger$.  Center: the $1700$ clean observations sampled from $(\calG_h)_\#\rho_\theta^\dagger$. Right: the empirical observed data distribution, comprising the same clean observations and $300$ samples from three off-range Gaussian components.  Blue points are clean PDE outputs and red points represent noisy samples that are out of the range.}}
    \label{fig:pde-over-data-generation}
\end{figure}

To reconstruct using KL divergence, we should note that this metric places strong requirement on the support of the distribution, so for running the gradient flow, we use a Gaussian kernel $K_\varepsilon$ as a filter. This is to set
\begin{equation*}
\calD_\phi (\calG_{\#} \rho_\theta \| \rho_y)\rightarrow \int_{\mathbb R^2}
r_{\varepsilon,N}\log\frac{r_{\varepsilon,N}}{q_{\varepsilon,N}}\,\rd y\,.
\end{equation*}
where
\[
r_{\varepsilon,N}=K_\varepsilon*\left((\calG_h)_\#\rho_N\right),
\qquad q_{\varepsilon,N}=K_\varepsilon*\rho_{y,N}\,.
\]
Deploying gradient flow for this objective function leads to 
\[
\dot\theta_i=-D\calG_h(\theta_i)^\top\nabla_y
\left[K_\varepsilon*\left(\log\frac{r_{\varepsilon,N}}
{q_{\varepsilon,N}}+1\right)\right]\!\left(\calG_h(\theta_i)\right).
\]
We numerically set $\varepsilon=8\times10^{-4}$ and run this update using 300 projected Armijo steps with trial step $2\times10^{-3}$. The objective decreases monotonically from $0.45$ to $0.16$, and we provide numerical results in Figure~\ref{fig:pde-particle-over-KL}. 

 {The left panel compares the initial parameter distribution, the final particle-flow reconstruction, and the prescribed theoretical density  $\rho_\theta^*$ which we obtain from the analytical form of $\mathcal G$ and Theorem~\ref{thm:phi_divergence}. The reconstruction captures both modes of $\rho_\theta^*$ and is substantially closer to it than the initial distribution. The right panel gives the analogous comparison in observation space. Since the range of the forward operator $\calG$ is a one-dimensional curve, we express both pushforward laws as densities with respect to the normalized arclength $s$ on this curve. The blue histogram represents the reconstructed data distribution at the final iteration, while the dashed black histogram is obtained using the theory. We observe that the reconstructed density does not take into account the data corruption. In particular, the recovered $\rho_y$ does not show any sign of inheriting mass from the corresponding three Gaussian perturbations. Their agreement provides a direct variational check that the particle flow recovers the on-range component of the observations, consistently with Theorem~\ref{thm:phi_divergence}.}


\begin{figure}[ht!]
    \centering
    \subfloat[KL reconstruction of $\rho_\theta$]{\includegraphics[height=0.23\textheight]{ 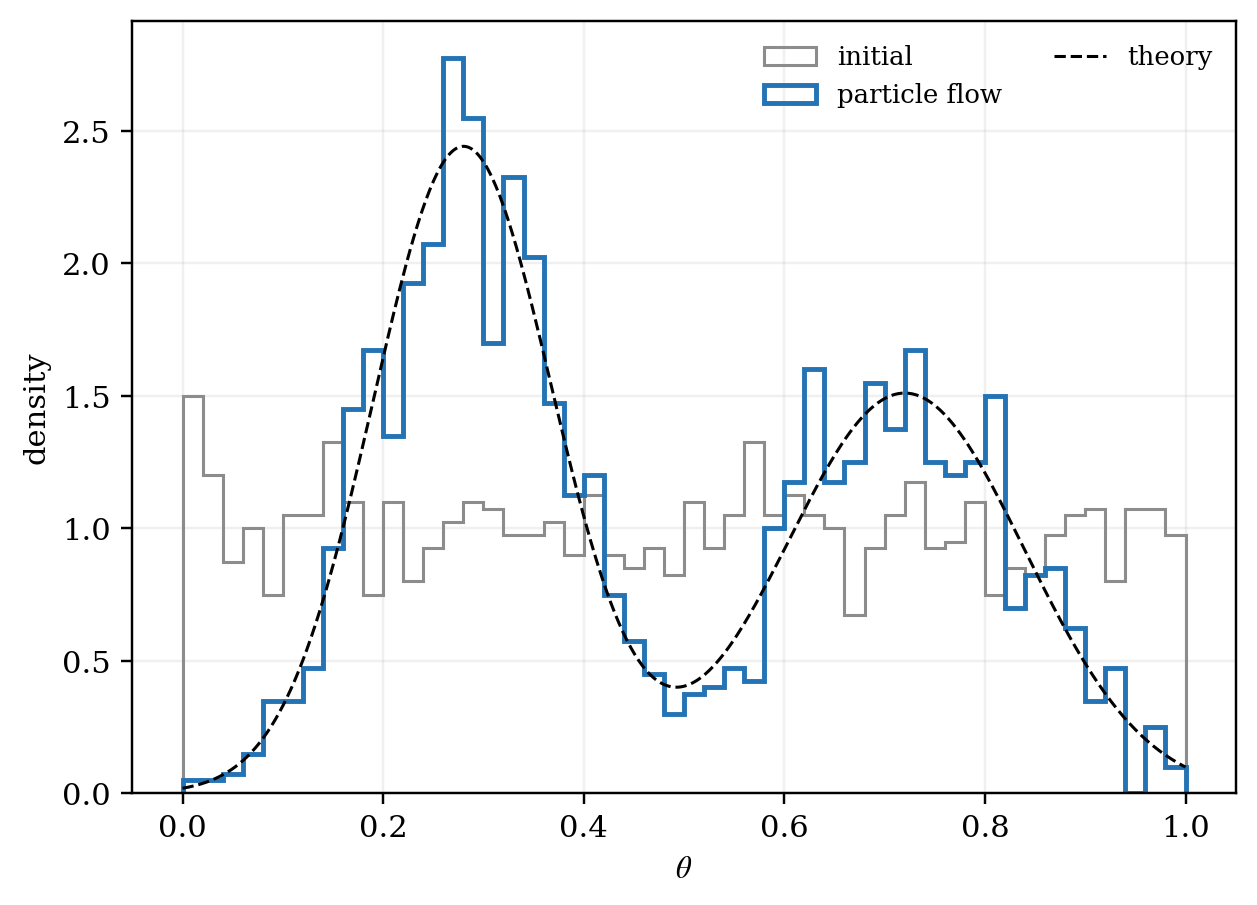}
    \label{fig:pde-particle-over-kl-theta}}
    \hfill
    \subfloat[KL reconstruction of $\rho_y$]{\includegraphics[height=0.23\textheight]{ 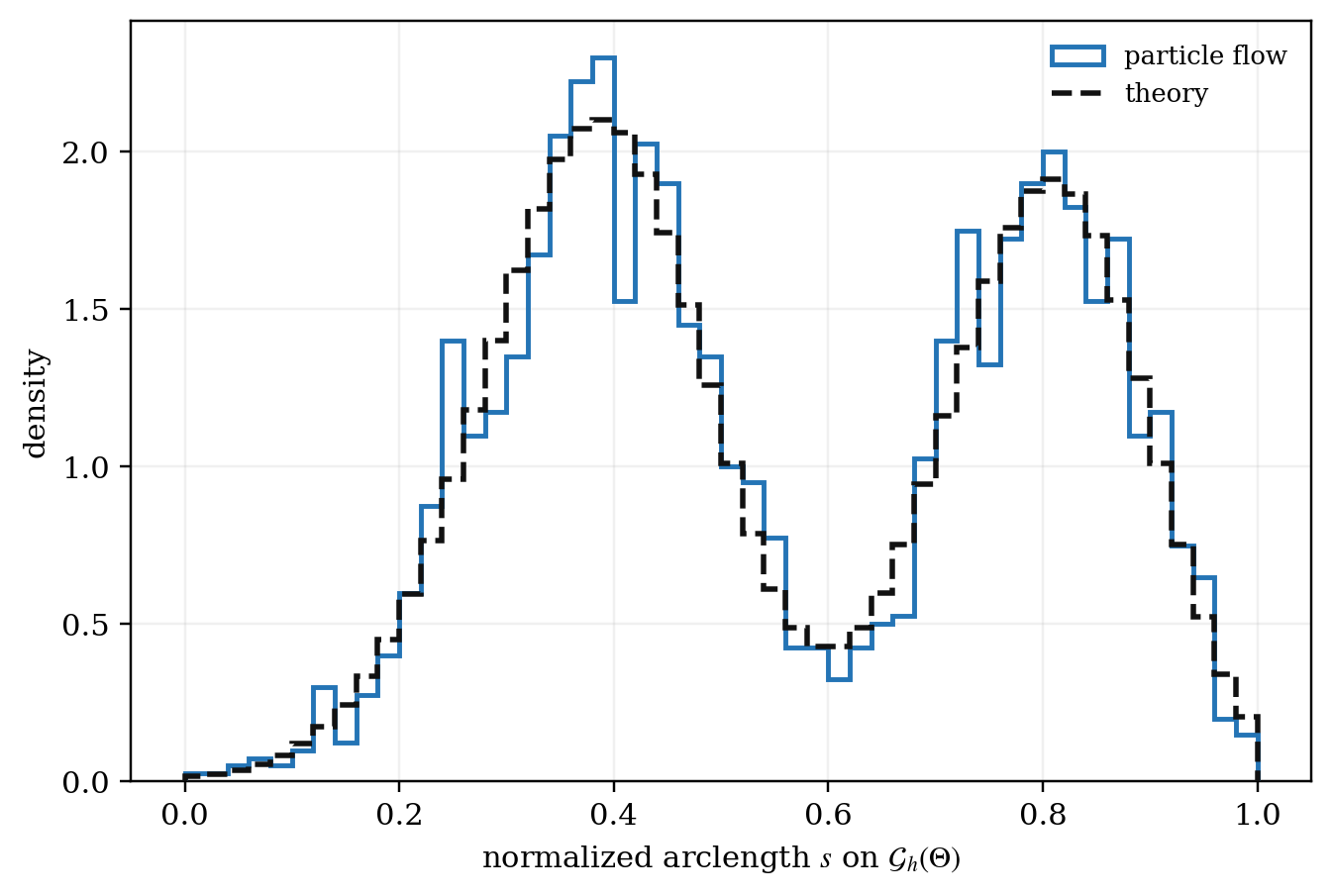}    \label{fig:pde-particle-over-KL-y}}
    \caption{ {Overdetermined particle reconstructions using KL divergence. Left: the gray histogram is the initial parameter distribution, the blue histogram is the final particle-flow reconstruction, and the dashed black curve is the prescribed theoretical parameter density. Right: the corresponding pushforward laws on the PDE range, expressed as densities with respect to normalized arclength $s$. The blue histogram is the pushforward of the final particle-flow reconstruction, and the dashed black histogram is obtained by our theory. We observe that the reconstructed density does not take into account the data corruption. In particular, $\rho_y$ does not show any sign of inheriting mass from the corresponding three Gaussian perturbations.}\label{fig:pde-particle-over-KL}}
\end{figure}

We now study the reconstruction using Wasserstein metric, this is to set
\[
\wass (\calG_{\#} \rho_\theta,  \rho_y)=\min_{\sigma\in\mathfrak S_N}
\frac1N\sum_{i=1}^N
\left|\calG_{\rm over}(\theta_i)-y_{\sigma(i)}\right|^2
\]
where we used the empirical version of~\eqref{eq:W_min}. To run gradient flow using this objective function, in each iteration we need to compute a new optimal transport problem between $\{\calG_{\rm over}(\theta_i)\}_{i=1}^N$ and $\{y_i\}_{i=1}^{N}$. Denoting by $\sigma_k$ the optimal plan at $k$-th iteration, we use the particle gradient flow
\[
\dot\theta_i=-2D\calG_h(\theta_i)^\top
\big(\calG_h(\theta_i)-y_{\sigma_k(i)}\big)
\]
and apply the Armijo line search condition to find the step size for obtaining an update. The objective decreases monotonically from $1.38\times10^{-5}$ to $1.22\times10^{-5}$ and we plot the reconstruction in Figure~\ref{fig:pde-particle-over-w2}. 

In the left panel, the gray histogram is the initial parameter distribution, the orange histogram is the final particle-flow reconstruction, and the dashed black histogram is obtained from the theory. In the right panel, the corresponding laws on the PDE range are expressed as densities with respect to normalized arclength $s$ on the range. The orange histogram is the particle-flow pushforward, and the dashed black histogram is obtained from the theory. The close agreement in both parameter and observation spaces therefore numerically verifies that minimizing the empirical Wasserstein objective recovers the projected observed law predicted by Theorem~\ref{thm:marginal}. As opposed to Figure~\ref{fig:pde-particle-over-kl-theta}, this reconstruction captures the Gaussian corruption. We observe the first peak corresponding to the leftmost Gaussian corruption, and the reconstructed distribution is skewed to the right, reflecting the projection of mass from the second and third peaks onto the range. 

\begin{figure}[ht!]
    \centering
    \subfloat[$W_2$ reconstruction of $\rho_\theta$]{\includegraphics[height=0.23\textheight]{ 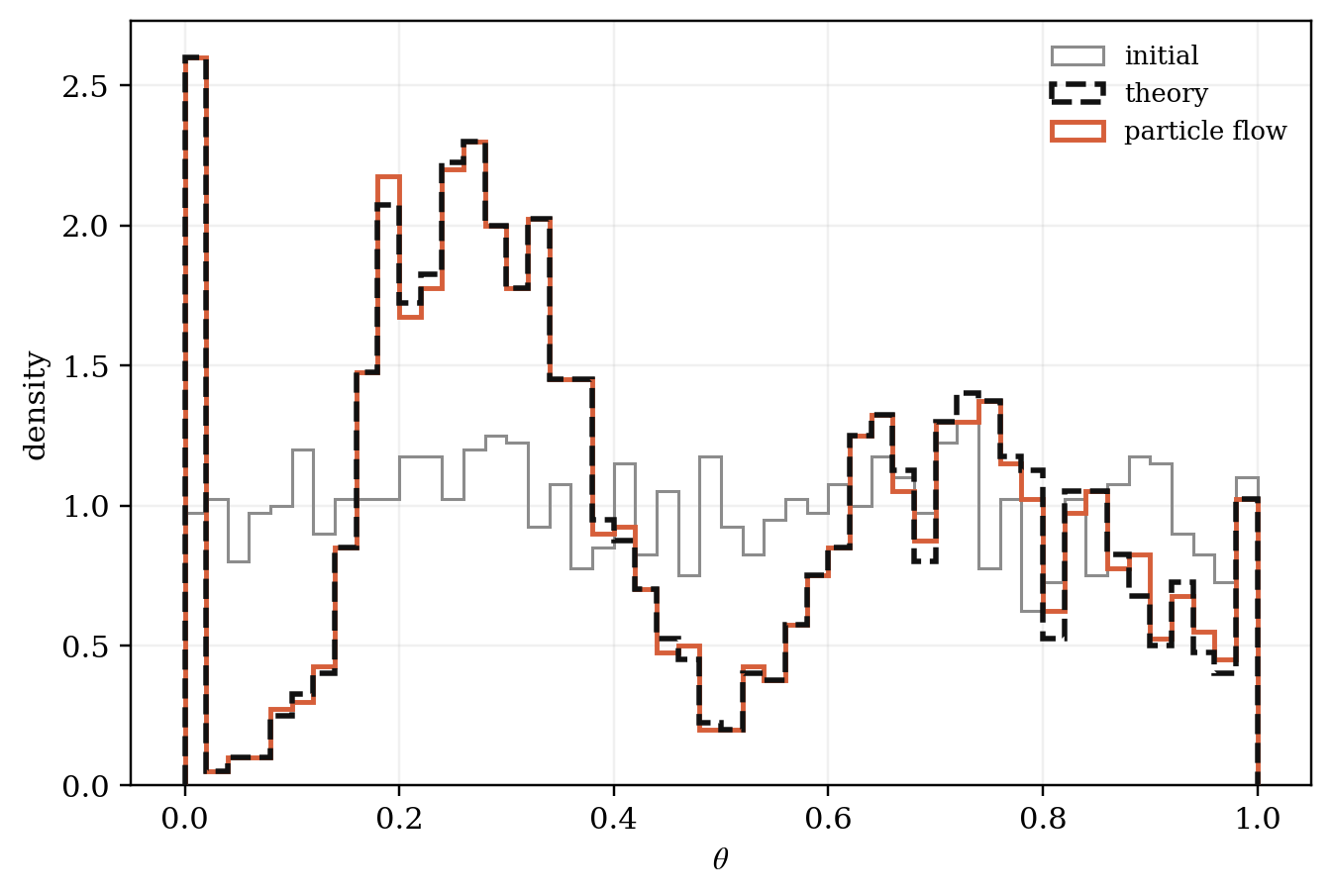}
    \label{fig:pde-particle-over-w2-theta}}
    \hfill
    \subfloat[$W_2$ reconstruction of $\rho_y$]{\includegraphics[height=0.23\textheight]{ 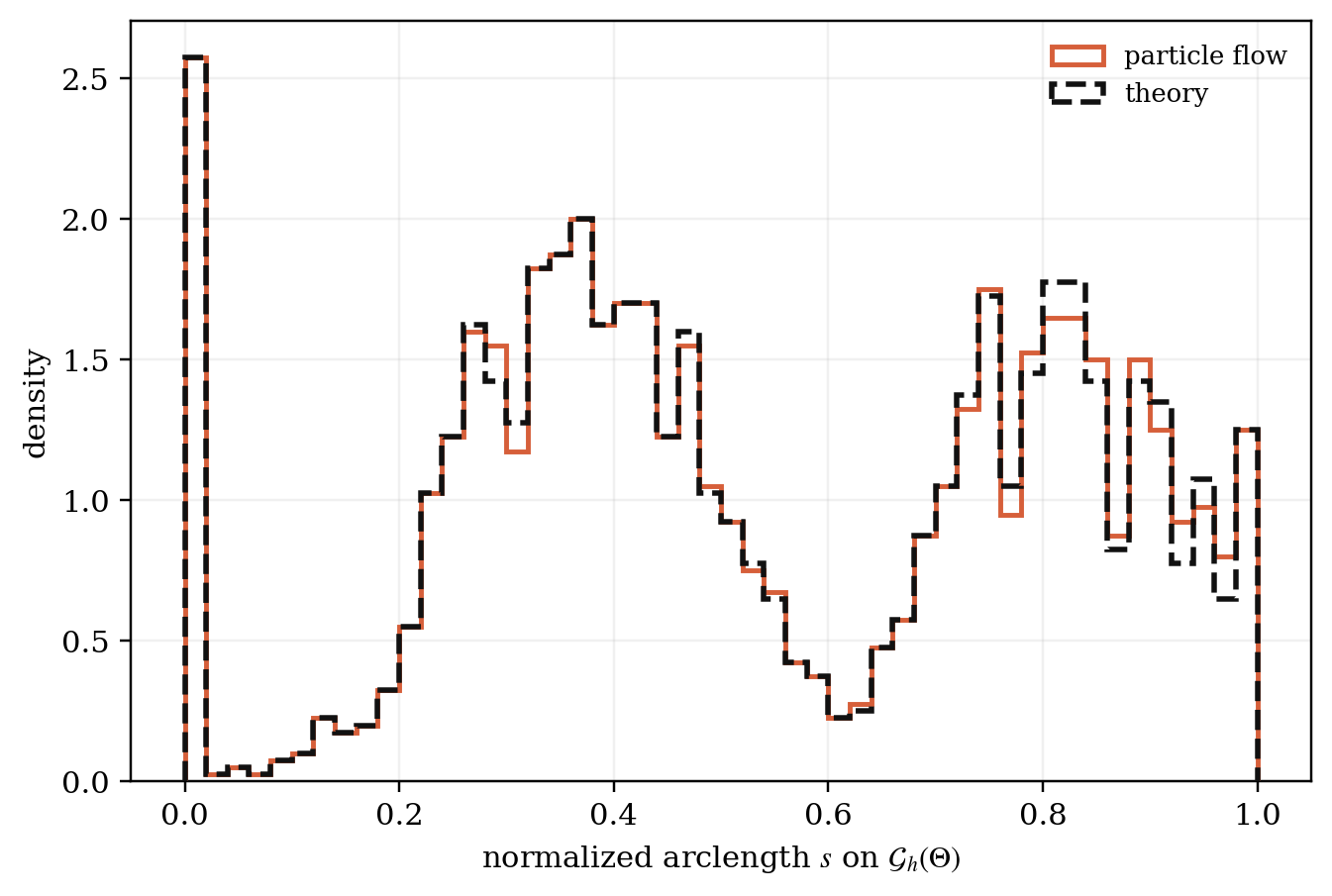}    \label{fig:pde-particle-over-w2-y}}
    \caption{ {Overdetermined particle reconstructions using Wasserstein distance. Left: the gray histogram is the initial parameter distribution, the orange histogram is the final exact-rematching particle-flow reconstruction, and the dashed black histogram is obtained from theory. Right: the orange histogram is the particle-flow pushforward and the dashed black histogram is obtained from theory. As opposed to Figure~\ref{fig:pde-particle-over-kl-theta}, this reconstruction in Figure~\ref{fig:pde-particle-over-w2-theta} captures the Gaussian corruption. We observe the first peak corresponding to the leftmost Gaussian corruption, and the reconstructed distribution is skewed to the right, reflecting the projection of the mass from the second and third peaks onto the range. Likewise, the reconstructed data distribution $\rho_y$ retains the effects of data corruption, as evidenced by the left peak and rightward skew in Figure~\ref{fig:pde-particle-over-w2-y}, in sharp contrast to the distribution shown in Figure~\ref{fig:pde-particle-over-KL-y}.} }  \label{fig:pde-particle-over-w2}
\end{figure}

\subsection{Underdetermined reconstruction}\label{sec:numerics-under}
We choose Case 2 for our simulation to demonstrate the underdetermined reconstruction. We use two parameters to parameterize $A_{\boldsymbol{\theta}}$ with $\boldsymbol{\theta}=(\theta_1,\theta_2)$, and measure one observation $\calG_{\rm under}(\boldsymbol\theta)$ as in~\eqref{eq:pde-under-map}.

Analytical calculations in~\eqref{eq:pde-under-map} suggest that the measurement $y=\calG_{\rm under}(\boldsymbol\theta)$ takes on the same value on each level set of $s=\theta_1+\theta_2$ . On a fixed level set $s=s_0$, $\theta_1=\theta_2=s_0/2$ is the least norm solution. According to Theorem~\ref{thm:KL-under} and Theorem~\ref{thm:moment_under}, the solution should be either a constant along each $s$-level set (for the case of KL) or concentrated along the diagonal.

Numerically, we provide $\{y_i\}_{i=1}^{2000}$ data points, for any given datum $y_i$, we first invert it for $s_i$ and run a constrained Wasserstein gradient flow within this level set for the recovery of $\boldsymbol{\theta}$. Writing $\boldsymbol\theta_i=(x_i,s_i-x_i)$,
\begin{itemize}
    \item minimizing the entropy calls for reflected Brownian motion
    \[
x_i^{k+1}=\operatorname{refl}_{I_i}\!\left(x_i^k+\sqrt{\Delta t}\,\xi_i^k\right),
\qquad \Delta t=5\times10^{-4},
\qquad \xi_i^k\sim\mathcal N(0,1),
\]
\item minimizing least-norm calls for tangential projection of $-\nabla|\boldsymbol\theta|^2$:
\[
x_i^{k+1}=\operatorname{refl}_{I_i}\!\left(x_i^k+\Delta t(s_i-2x_i^k)\right),
\qquad \Delta t=0.1,
\]
\end{itemize}
where $I_i=[\max(0,s_i-1),\min(1,s_i)]$ and $\operatorname{refl}_{I}$ is the reflection operator onto an interval $I=[a,b]$ to keep a particle inside $I$.

We provide the initial data that is far away from the final solution; see the left panels of Figure~\ref{fig:pde-particle-under-entropy-moment}. For the entropy case, we start with a singular initial distribution, while for the moment case, we start with a uniform distribution along each level set. The final entropy- and moment-based inversions are presented in the right panels of Figure~\ref{fig:pde-particle-under-entropy-moment}. These simulations confirm our theoretical prediction.



\begin{figure}[h!]
    \centering
    \includegraphics[width=0.9\linewidth]{ 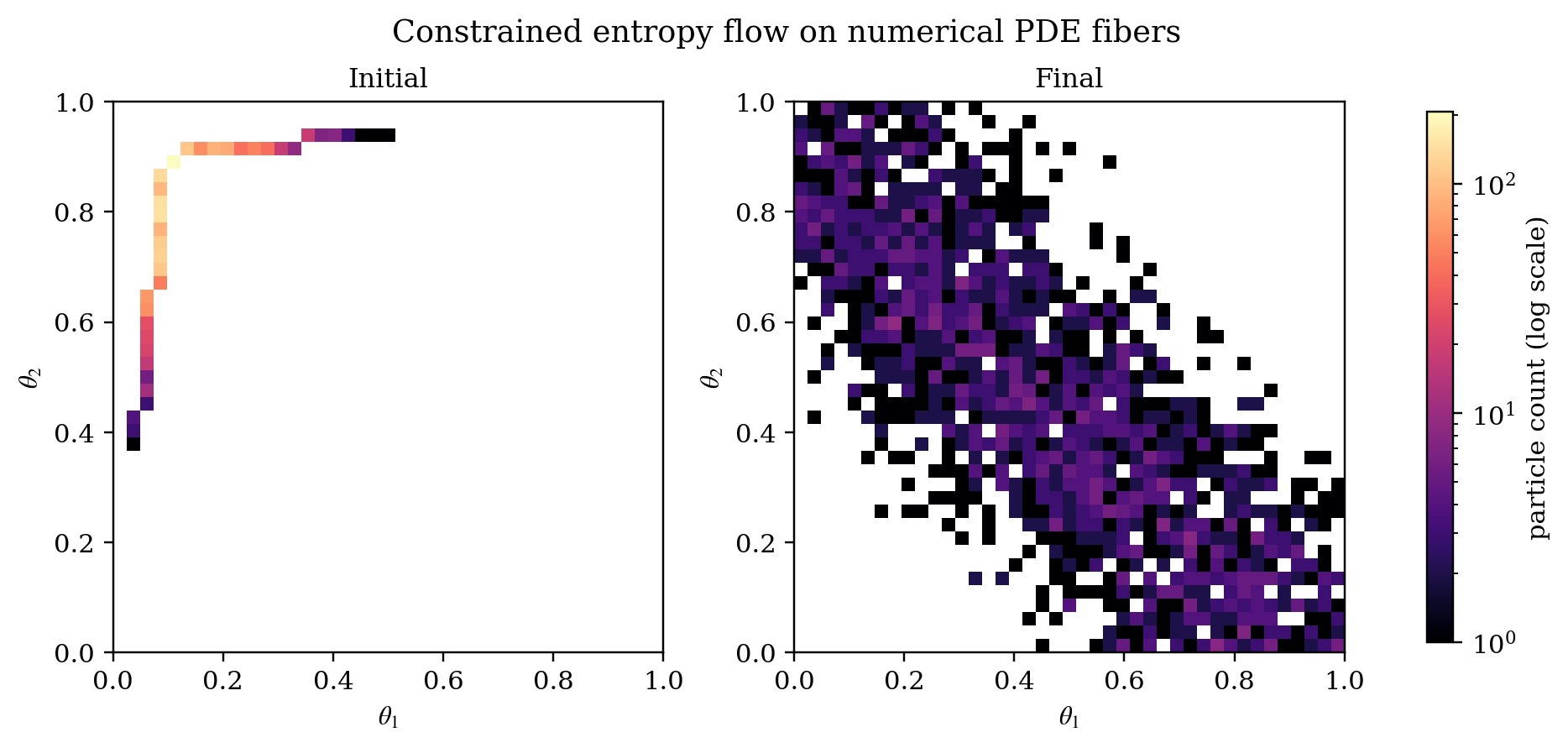}\\
    \includegraphics[width=0.9\linewidth]{ 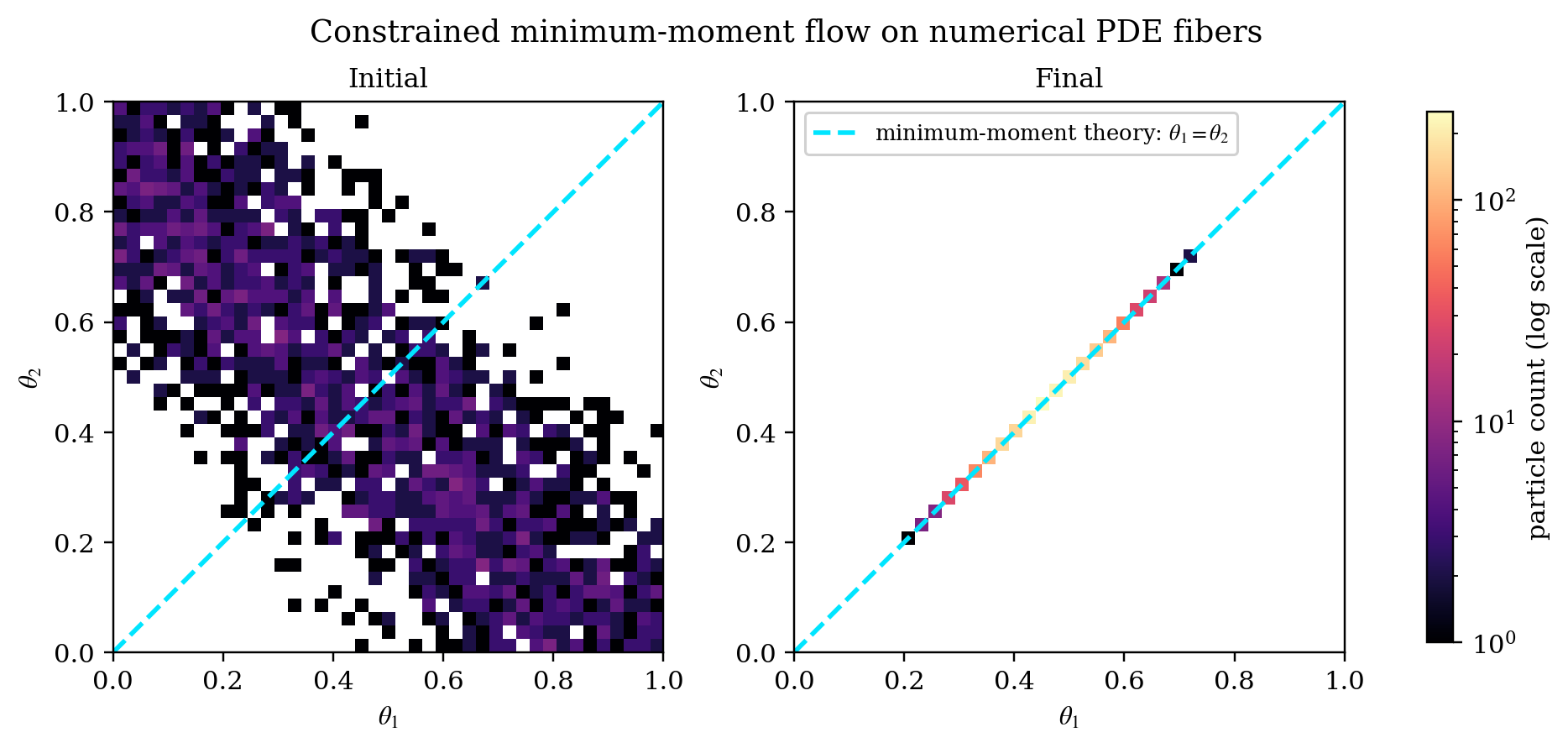}
    \caption{ {The initial guess is prepared far from the final solution, but in both cases, constrained Wasserstein gradient flow using $N=2000$ particles successfully recovered the analytical prediction: the entropy solution provided a uniform distribution over each level set, while the second-moment solution was highly concentrated along the diagonal.}}
    \label{fig:pde-particle-under-entropy-moment}
\end{figure}



\subsection{Regularized reconstruction}\label{sec:numerics-reg}

In this subsection, we use Case 2 to provide a simulation, and $N=2000$ data points provided using~\eqref{eq:pde-under-map} with added independent noise distributed uniformly on $[-10^{-3},10^{-3}]$.

To numerically execute the regularized case, we run Wasserstein gradient flow and use a reflection operator to enforce $\boldsymbol{\theta}$ to stay within the domain. This amounts to
\begin{itemize}
    \item For the KL-KL regularization:
    \begin{align*}
\boldsymbol\theta_i^{k+1}=\boldsymbol\theta_i^k +\Delta t\bigg[-D\calG_h(\boldsymbol\theta_i^k)^\top
\partial_y\Psi_{\varepsilon,k}\big(\calG_h(\boldsymbol\theta_i^k)\big)
+\alpha\nabla\log\calM(\boldsymbol\theta_i^k)\bigg]+\sqrt{2\alpha\Delta t}\,\boldsymbol\xi_i^k ,
\end{align*}
where $\boldsymbol\xi_i^k\sim\mathcal N(0,\mathsf I_2)$, $\Psi_{\varepsilon,k}=K_\varepsilon*\left(\log\frac{r_{\varepsilon,k}}{q_\varepsilon}+1\right)$ and we set the blob bandwidth $\varepsilon=8\times10^{-4}$, $\Delta t=5\times10^{-5}$, and run the updates using $8{,}000$ steps.
\item For the $W_2^2$-moment case:
\[
\dot{\boldsymbol\theta}_i=-2D\calG_h(\boldsymbol\theta_i)^\top
\big(\calG_h(\boldsymbol\theta_i)-y_{\sigma_k(i)}\big)
-2\alpha\boldsymbol\theta_i.
\]
and apply the Armijo line search for updating.
\end{itemize}
In Figure~\ref{fig:pde-particle-reg-kl} we present our numerical findings. In both cases, the simulations confirm the theoretical predictions.


\begin{figure}
    \centering
    \includegraphics[width=0.9\linewidth]{ 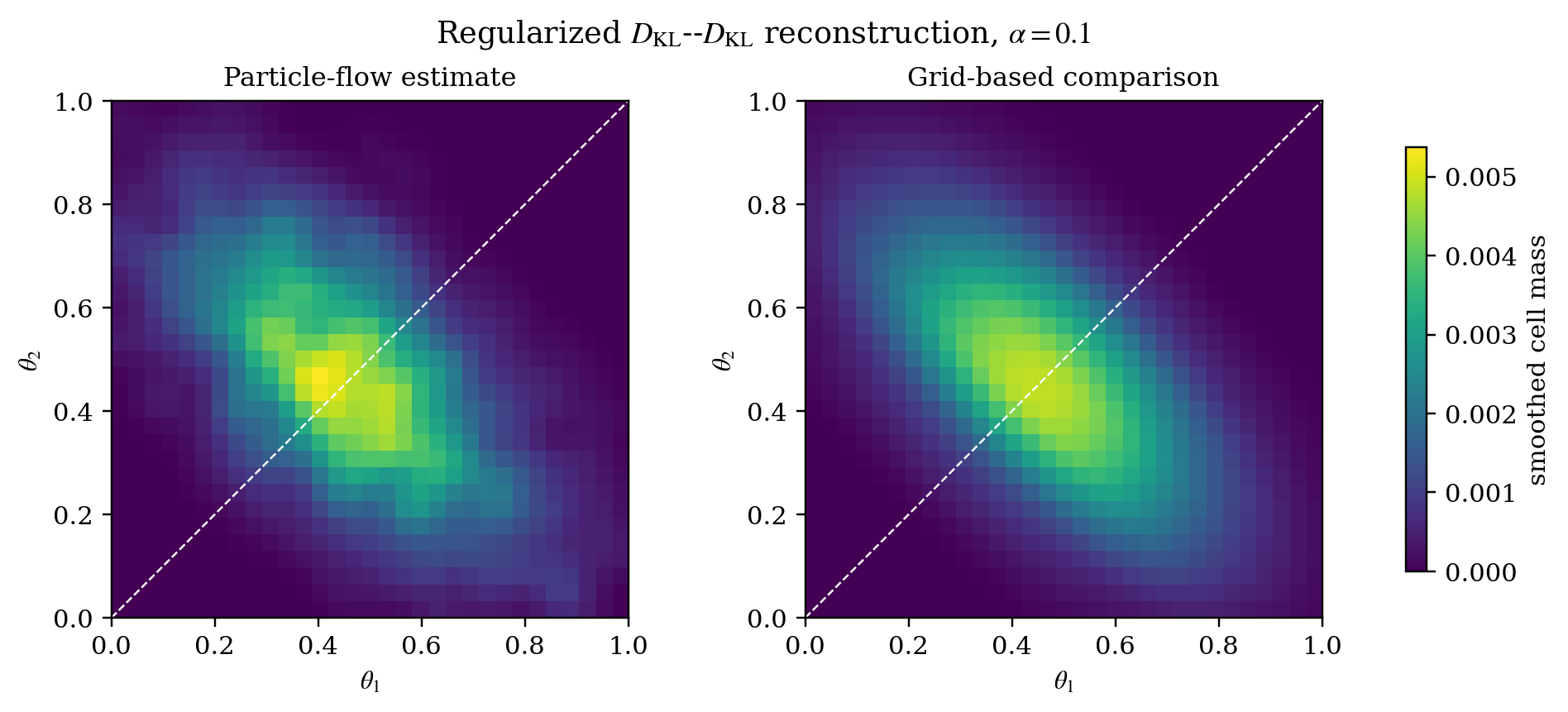}
    \includegraphics[width=0.9\linewidth]{ 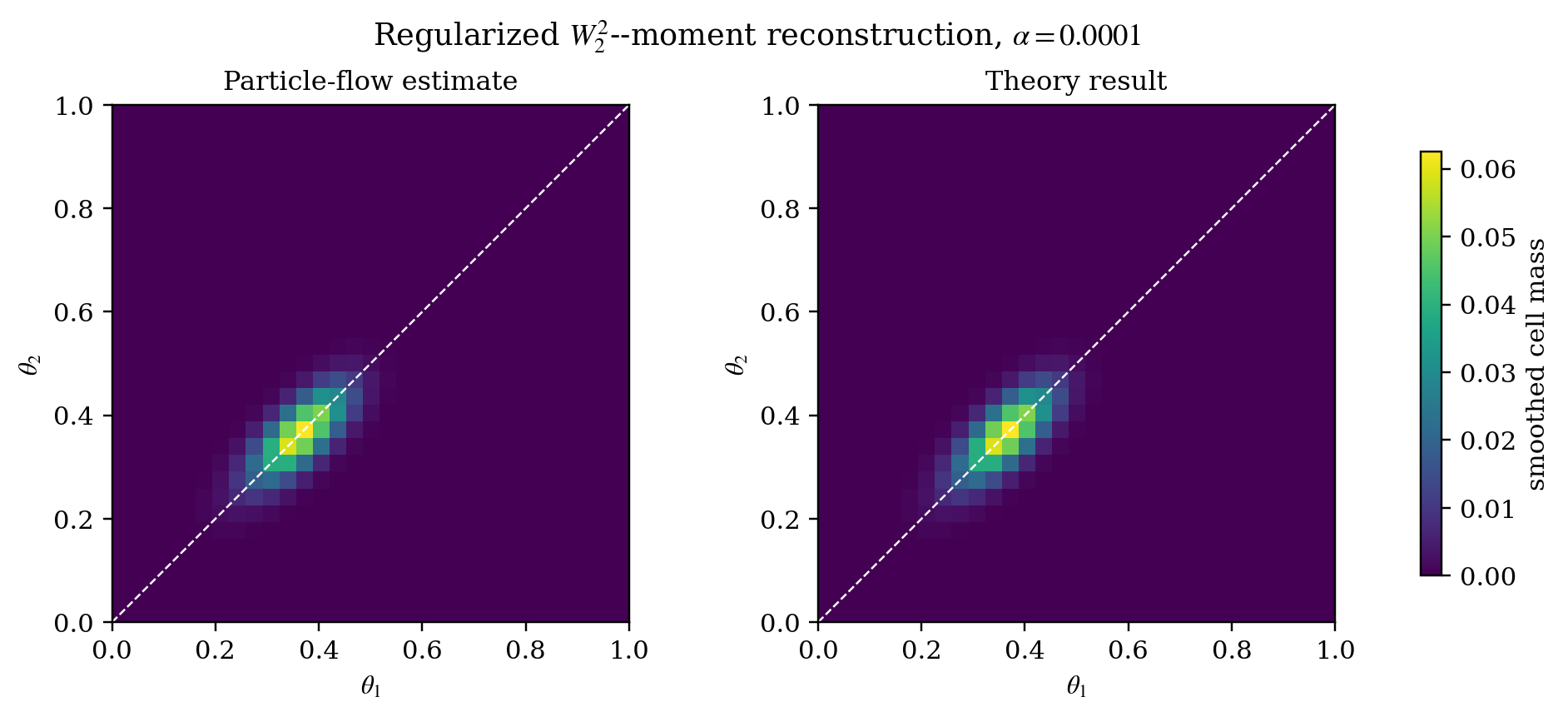}
    \caption{ {Top row shows KL--KL reconstruction using regularization parameter $\alpha=0.1$ compared with  reference minimizer~\eqref{312_revise} based on an analytical solution in Theorem~\ref{reg-KL}. Bottom row shows the $W_2^2$--moment reconstruction at $\alpha=10^{-4}$. The left panel is the numerical particle-flow result, and the right panel is the independently computed theoretical result in Theorem~\ref{thm:moment_regularizer}.}}
    \label{fig:pde-particle-reg-kl}
\end{figure}

We further examine the effect of the regularization coefficient on the reconstruction. The left panel of Figure~\ref{fig:pde-particle-reg-paths} shows the reconstruction result when KL-KL is used. As $\alpha$ increases from $0.01$ to $1$, the objective shifts its focus from matching $\rho_y$ to matching $\rho_{\boldsymbol{\theta}}$. Consequently, the estimated KL for $\rho_{\boldsymbol{\theta}}$ decreases, that for $\rho_y$ increases. The right panel shows the regularization effect on the $W_2$-moment pair. As $\alpha$ increases, the objective function places more weight on the second moment, and as a consequence, the empirical Wasserstein distance increases and the moments decrease, so particles become concentrated along the diagonal.

\begin{figure}
    \centering
    \includegraphics[width=0.98\linewidth]{ 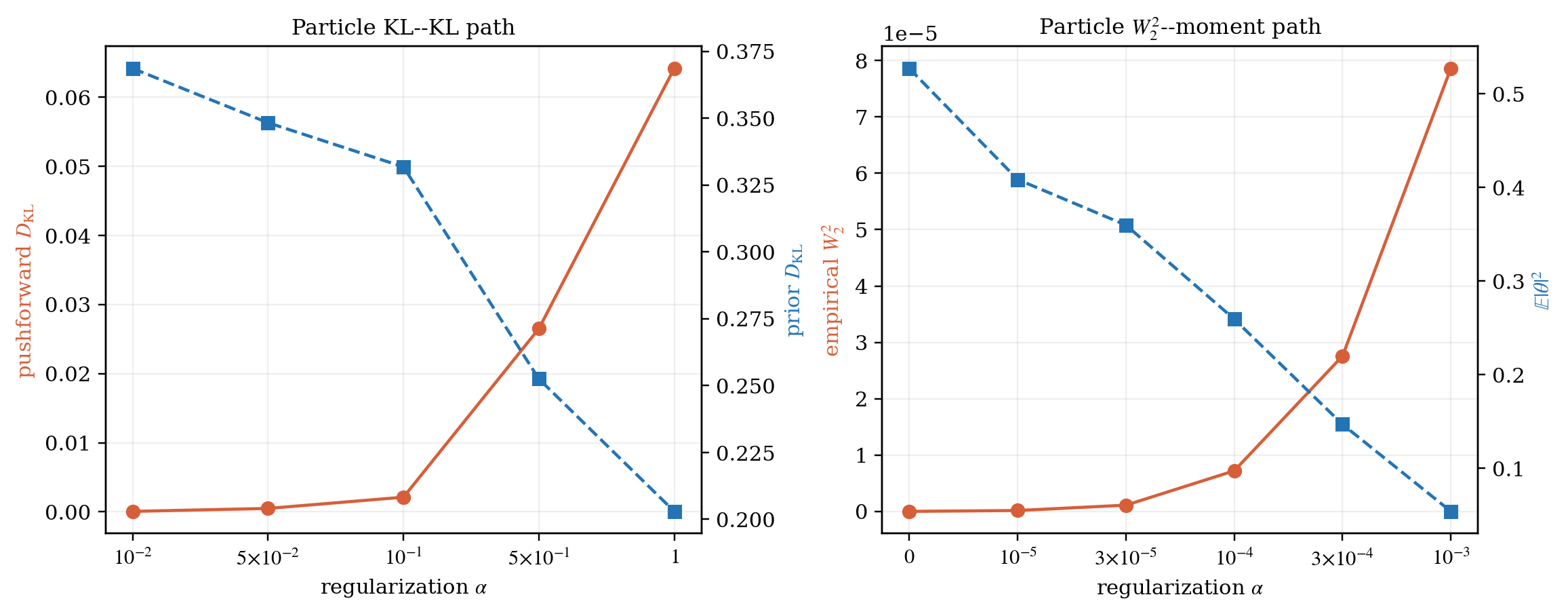}
    \caption{ {Particle regularization paths.  In both cases, stronger regularization moves the parameter measure toward the distribution favored by the penalty while increasing the data mismatch.}}
    \label{fig:pde-particle-reg-paths}
\end{figure}

}

\section*{Acknowledgments}
The authors thank Prof.~Youssef Marzouk and Prof.~L\'ena\"ic Chizat for comments and suggestions on an earlier version of this work.

\bibliography{sample.bib, ref.bib}
\bibliographystyle{siam}

\end{document}